\documentclass[reqno]{amsart}
\usepackage{amssymb,latexsym,color,}
\usepackage{amsmath}
\usepackage{amsthm}
\usepackage{graphicx}
\usepackage{palatino,mathpazo}
\usepackage{setspace}
\usepackage{titletoc}
\numberwithin{equation}{section}
\newtheorem{theorem}{Theorem}[section]

\newtheorem{lemma}[theorem]{Lemma}

\newtheorem{definition}[theorem]{Definition}

\theoremstyle{definition}

\def\t{\widetilde}

\def\XXint#1#2#3{{\setbox0=\hbox{$#1{#2#3}{\int}$}
     \vcenter{\hbox{$#2#3$}}\kern-.5\wd0}}

\def\p{\partial}

\newcommand{\RN}{\mathbb{R}^2}
\newcommand{\intr}{\int_{\mathbb{R}^2}}

\newcommand{\e}{\varepsilon}

\newcommand{\sscp}{\scriptscriptstyle}

\begin{document}
\title[Uniqueness for bubbling solutions with collapsing singularities]
{Uniqueness for bubbling solutions with collapsing singularities}

\author{Youngae Lee}
\address{Youngae ~Lee,~National Institute for Mathematical Sciences, 70 Yuseong-daero 1689 beon-gil, Yuseong-gu, Daejeon, 34047, Republic of Korea}
\email{youngaelee0531@gmail.com}
\author{ Chang-Shou Lin}
\address{ Chang-Shou ~Lin,~Taida Institute for Mathematical Sciences, Center for Advanced Study in
Theoretical Sciences, National Taiwan University, Taipei 106, Taiwan}
\email{cslin@math.ntu.edu.tw}

\date{\today}
\begin{abstract}The seminal work \cite{bm} by  Brezis and Merle showed that the bubbling solutions of the mean field equation  have the property of  mass concentration.  Recently, Lin and Tarantello  in \cite{lt} found that  the  "bubbling implies mass concentration" phenomena might not hold  if there is a   collapse of singularities.  Furthermore, a sharp estimate \cite{llty} for the bubbling solutions has been obtained. In this paper, we prove  that  there exists at most one sequence of bubbling solutions if the collapsing singularity occurs.  The main difficulty comes from that after re-scaling, the difference of two solutions locally converges to an element in  the kernel space of  the linearized operator. It is well-known that the kernel space is  three dimensional.
   So the main technical ingredient of  the proof is  to show that the limit after re-scaling is orthogonal to the  kernel space.  
\end{abstract}

 \maketitle

\section{Introduction}
We are concerned with the following mean field type equation:
\begin{equation}
\label{mfd}\left\{\begin{array}{l}
\Delta_M u+\rho\left(\frac{h_*e^{u}}{\int_Mh_*e^{u}{d}v_g}-1\right)=4\pi\sum_{q_i\in S}\alpha_{i}(\delta_{q_i}-1)\ \textrm{in} \ M,
\\  \int_M u  {d}v_g=0,
\end{array}\right.\end{equation}
where $(M,ds)$ is a compact Riemann surface, $\rho>0$, $dv_g$ is the volume form, $\Delta_M$ is the Laplace-Beltrami operator on $(M,ds)$,  $S\subseteq M$ is a finite set of distinct points $q_i$, $\alpha_{q_i}>-1$,
  and $\delta_{q_i}$ is the Dirac measure at $q_i$. The point $q_i$ with Dirac measure is called vortex point or singular source.   Throughout this paper, we always assume that $|M|=1$, $h_*>0$ and $h_*\in C^{3}(M)$.  The equation \eqref{mfd} arises in various different fields.  In conformal geometry, \eqref{mfd} is related  to the Nirenberg problem of finding  prescribed Gaussian curvature
if $S=\emptyset$, and the existence of a positive constant curvature metric with conic
singularities if $S\neq\emptyset$ (see  \cite{ t} and the references therein).   The equation \eqref{mfd} is also related to the self-dual equation of the relativistic Chern-Simons-Higgs model.
For the recent developments related to \eqref{mfd}, we refer to the readers to   \cite{bamar,   bt1, bclt, bm,  ck,cl1, cl2, cl3, cl4,  l1,   ly1,  lw0, lw1, m1, m2,   m4,   nt2,t, y,y2} and references therein.

   In the  seminal work \cite{bm} by Brezis and Merle,  the blow up behavior of solutions for \eqref{mfd} has been studied as follows:

\medskip
\noindent {\bf{Theorem A}}. \cite{bm,ls,bt1}  {\em  Given fixed each vortex point $q_i\in S$, suppose   $\alpha_i \in \mathbb{N}$, $i=1,\cdots,N$. We assume that $h^*$ is a positive smooth function. Let  $u^*_k$ be a sequence of blow up solutions for \eqref{mfd}, that is: $\text{max}_M\ u_k^* \rightarrow +\infty$ as $k \rightarrow +\infty$. Then there is a non- empty finite set $\mathcal{B}$ (blow up set)
such that,
 \[\rho\frac{h^*e^{u_k^*}}{\int_Mh^*e^{u_k^*}{d}v_g}\to\sum_{p\in \mathcal{B}}\beta_p\delta_p,\ \ \textit{where}\ \ \beta_p\in 8\pi\mathbb{N}.\]}

\medskip
\indent   For equation \eqref{mfd}, we call $\rho\frac{h^*e^{u^*}}{\int_Mh^*e^{u^*}{d}v_g}$ the \textit{mass distribution} of the solution $u^*$. Following this terminology,  Theorem A states that:  When the vortex points are not collapsing, the mean field equation possesses the property of the so-called "blow up solutions has the mass concentration property".
The version of Theorem A for the following Chern-Simons-Higgs (CSH) equation was also proved in \cite[Theorem 3.1]{ck}:
\begin{equation}\label{csh}
  \Delta_{\mathbb{T}} u+\frac{1}{\varepsilon^2}e^u(1-e^u)=4\pi\sum_{j=1}^N\delta_{p_j}\ \ \textrm{in}\ \ \mathbb{T},
\end{equation}
where $\varepsilon>0$
 is a  small parameter and $\mathbb{T}$ is a flat torus.
The equation (\ref{csh}) was derived from the CSH model  to describe vortices in high temperature
superconductivity, and has been extensively studied during few decades.  We refer the readers to  \cite{JW, T2, y, ck, CKL, ly1, ly2, nt2} and references therein.
Among them, Lin and Yan in \cite{ly2} proved the local uniqueness result of bubbling solutions for \eqref{csh}, that is, if $u_{\varepsilon,1}$ and $u_{\varepsilon,2}$ are two sequence of bubbling solutions blowing up at the same points under some non-degenerate condition, then $u_{\varepsilon,1}=u_{2,\varepsilon}$ for small $\varepsilon>0$.  By applying the idea in \cite{ly2}, the local uniqueness result of bubbling solutions of \eqref{mfd} was obtained in \cite{bjly}.  We note that the works \cite{ly2,bjly} are concerned with the local uniqueness of bubbling solutions when the vortex points are not collapsing.

  However,  there is a big  difference  when the collapsing singularities are considered. First, Lin and Tarantello in \cite{lt} observed a new phenomena such that  blow-up solutions with collapsing singularities might not have  the "concentration" property of its mass distribution. The general version was studied in \cite{llty}. To describe the results, let us consider the following equation:
\begin{equation}
\label {0.3}
\left\{\begin{array}{l}
\Delta_M u^*_t+\rho\left(\frac{h^*e^{u^*_t}}{\int_Mh^*e^{u^*_t}{d}v_g}-1\right)\\=4\pi\sum_{i=1}^d\alpha_i(\delta_{q_i(t)}-1)+ 4\pi\sum_{i=d+1}^N\alpha_{i}(\delta_{q_i}-1)\ \mbox{in} \ M,\\
 \int_M u^*_t {d}v_g=0,
\end{array}\right.
\end{equation}
where $\lim_{t\to0}q_i(t)= \mathfrak{q}\notin\{q_{d+1}, \cdots, q_N\}$, $\forall\ i = 1, \cdots , d$ and $q_i(t) \neq q_j(t)$ for $i \neq j \in \{ 1,\cdots,d \}$.  Then the following holds:
\medskip

\noindent {\bf{Theorem B}}. \cite{lt,llty}  {\em  Assume $\alpha_i\in \mathbb{N},~1\leq i\leq N.$ Let $u_t^*$ be a sequence of blow up solutions of (\ref{0.3}) with $\rho \notin 8\pi \mathbb{N}.$ Then $u_t^*$ blows up only at $\mathfrak{q}\in M$. Furthermore, there exists a function $w^*$ such that
$$u_t^*\to w^*~\mathrm{in}~C_{\mathrm{loc}}^2(M\setminus\{\mathfrak{q}\})$$
as $t\to0$, and $w^*$ satisfies:
\begin{equation}
\label{0.8}\left\{\begin{array}{l}
 \Delta_M w^*+(\rho-8m\pi)\left(\frac{h^*e^{w^*}}{\int_Mh^*e^{w^*}}-1\right)
\\ =4\pi\left(\sum_{i=1}^d\alpha_i-2m\right)(\delta_q-1) + 4\pi\sum\limits_{i=d+1}^N \alpha_i (\delta_{q_i}-1) \ \ \textrm{in}\ \ M, \\  \int_M w^* {d}v_g=0,
\end{array}\right.
\end{equation}
for some $m\in\mathbb{N}$ with $1\leq m\leq[\frac12\sum_{i=1}^d\alpha_i]$ \footnote{$[x]$ stands for the integer part of $x$.} and $\rho > 8m\pi $.}

\medskip
\indent
So Theorem B tells us that the mass concentration does no longer hold  if the collapsing singularity occurs. Indeed, we have
$\lim_{t\to0}\int_M h^*e^{u_t^*}dv_g<+\infty$, which is different from the situation described in Theorem A.
We note that Theorem B could be improved provided that  the following nondegeneracy condition holds:
\begin{definition}
  A solution $w^*$ of (\ref{0.8}) is said non-degenerate, if the linearized problem
\begin{align}
\label{1.6}
\Delta_M \phi+(\rho-8m\pi)\frac{h^*e^{w^*}}{\int_Mh^*e^{w^*}{d}v_g}
\left(\phi-\frac{\int_Mh^*e^{w^*}\phi{d}v_g}{\int_Mh^*e^{w^*}{d}v_g}\right)=0,\quad \int_M\phi{d}v_g=0
\end{align}
only admits the trivial solution.
\end{definition}

Using the transversality theorem, we can always choose a positive smooth function $h^*$ such that $w^*$ is non-degenerate. See Theorem 4.1 in \cite{llwy}.
Based on the non-degeneracy assumption for (\ref{0.8}), some sharper estimates on the bubbling solutions of (\ref{0.3}) were obtained in \cite{llty} (see also section \ref{sec_pre} below).

For the simplicity, throughout this paper, we focus on  the case where the collapsing vortices are only {\em two}, that is,
\begin{equation}\label{main assume}
  d=2, \alpha_1=\alpha_2=1, \alpha_i\in\mathbb{N}\ \ \textrm{for}\ \ i=3,\cdots,N.
\end{equation}
Our main goal is to prove the local uniqueness of blow up solutions of \eqref{0.3} with collapsing singularities.
\begin{theorem}\label{thm_nonconcen} Assume that \eqref{main assume} holds and $\rho\notin8\pi\mathbb{N}$.  Suppose that  $u_{t,1}^{*}$ and $u_{t,2}^{*}$ are two blow up solutions of \eqref{0.3}. Assume that $u_{t,1}^{*}, u_{t,2}^{*}\to w^*$ in $ C_{\textrm{loc}}(M\setminus\{\mathfrak{q}\})$, where $w^*$ is a non-degenerate solution of (\ref{0.8}) with $m=1$. Then  $u_{t,1}^{*}=u_{t,2}^{*}$ for sufficiently small $t>0$.
\end{theorem}We remark that the study of blow up solutions of \eqref{0.3} with collapsing singularities is also important to  compute the topological degree for the  Toda system as noticed in \cite{llwy,llyz}, where the degree counting of the whole system is reduced to  computing the
degree of the corresponding shadow system (see  \cite[Theorem 1.4]{llwy}).  Thus  it is inevitable to   encounter with  the phenomena of collapsing singularities when we want to find a priori bound for solutions of  an associated shadow system.  For the details, we refer to the readers to \cite{llwy,llyz}.

In order to prove Theorem \ref{thm_nonconcen}, we need to analyze the asymptotic behaviour of
$\zeta_t=\frac{u_{t,*}^{(1)}-u_{t,*}^{(2)}}{\|u_{t,*}^{(1)}-u_{t,*}^{(2)}\|_{L^\infty(M)}}$. After a suitable scaling on small domain of $\mathfrak{q}$, $\zeta_t$ converges to an entire solution of the linearized problem associated to the Liouville equation:
\begin{equation}
\label{liouville}
\Delta v+ e^{v}=0\quad\textrm{in}\ \mathbb{R}^2,
\end{equation}where $\Delta=\sum_{i=1}^2\frac{\partial^2}{\partial x_i^2}$
denotes the standard Laplacian in $\mathbb{R}^2$.
A solution $v$ of \eqref{liouville} is completely classified \cite{cli1} such that
\begin{equation}\label{17}
v\left( z\right)=v_{\mu,a}(z) = \ln \frac{8e^{\mu}}{ ( 1+e^{\mu}  \vert z +a  \vert ^{2} ) ^{2}},  \quad \mu  \in \mathbb{R},\;a=(a_1,a_2) \in \mathbb{R}^2.
\end{equation}
The freedom in the choice of $\mu$ and $a$ is due to the invariance of equation \eqref{liouville} under dilations and translations.
The linearized operator $L$ relative to $v_{\sscp 0,0}$  is defined by,
\begin{equation}\label{entirelinear}L\phi:=\Delta \phi+\frac{8}{(1+|z|^2)^2}\phi\quad\textrm{in}\ \mathbb{R}^2.
\end{equation}
  In \cite[Proposition 1]{bp}, it has been proved that any kernel of $L$ is a linear combination of $Y_0$, $Y_1$, $Y_2$,
where
\begin{equation}
\begin{aligned}\left\{ \begin{array}{ll}\label{z01}Y_0(z):= \frac{1-|z|^2}{1+ |z|^2}=-1+\frac{2}{1+|z|^2}=\frac{\partial v_{\mu,a}}{\partial \mu}\Big|_{(\mu,a)=(0,0)},\\
Y_1(z):= \frac{z_1}{1+ |z|^2}=-\frac{1}{4}\frac{\partial v_{\mu,a}}{\partial a_1}\Big|_{(\mu,a)=(0,0)},
\\ Y_2(z):=\frac{z_2}{1+ |z|^2}=-\frac{1}{4}\frac{\partial v_{\mu,a}}{\partial a_2}\Big|_{(\mu,a)=(0,0)}.
\end{array}\right. \end{aligned}
\end{equation}
The orthogonality to $Y_1, Y_2$ can be obtained by applying a suitable Pohozaev-type identities as in \cite{ly2}.  However,  due to the non-concentration of mass, we meet a new difficulty to show  the orthogonality with $Y_0$. In order to overcome this obstacle, we apply the Green's representation formula with some delicate analysis. This idea comes from the recent work \cite{lly}. In \cite{lly}, it was proved that if $w^*$ is a non-degenerate solution of (\ref{0.8}), and the assumptions \eqref{main assume}  and $\rho\notin8\pi\mathbb{N}$ hold, then  there is a blow up solution $u_{t}^{*}$ of \eqref{0.3} such that $u_{t}^{*}\to w^*$ in $ C_{\textrm{loc}}(M\setminus\{\mathfrak{q}\})$.

This paper is organized as follows.
In section \ref{sec_pre}, we  review some known sharp estimates for blow up solutions of \eqref{0.3}.
In section \ref{sec_est}, we analyze the limit behavior of $\zeta_t$ in   $M\setminus\{\mathfrak{q}\}$ and
a small neighborhood of $\mathfrak{q}$ respectively. Finally, we  prove Theorem \ref{thm_nonconcen} by using suitable Pohozaev-type identities and Green's representation formula.

\section{Preliminary}\label{sec_pre}
Let $G(x,p)$ denote the Green's function for the Laplace Beltrami operator $\Delta_M$ on $M$, that is
\begin{equation} \label{def_g}
\Delta_M G(x,p)+(\delta_p-1)=0,~\quad\int_MG(x,p)d\sigma(x)=0.
\end{equation}
We recall the following assumption:
\begin{equation*}
  d=2, \alpha_1=\alpha_2=1, \alpha_i\in\mathbb{N}\ \ \textrm{for}\ \ i=3,\cdots,N.
\end{equation*}
Let $u_t^*$ be a sequence of blow up solutions of (\ref{0.3}) and $w^*$ be the limit of $u_t^*$ in Theorem B.
Set \begin{align}
\label{def-ut-function}
u_t(x)=u_t^*(x)+4\pi\sum_{i=1}^{2}G(x,q_i(t))+4\pi\sum_{i=3}^N\alpha_iG(x,q_i),
\end{align}
and
\begin{align}
\label{def-w-function}
w(x)=w^*(x) +4\pi\sum_{i=3}^N\alpha_iG(x,q_i).
\end{align}
We may choose a suitable coordinate centered at $\mathfrak{q}$ and \[\mathfrak{q}=0,\ \ q_1(t)=t\underline{e},\ \ q_2(t)=-t\underline{e},\ \ \textrm{ where}\ \ \underline{e}\ \ \textrm{ is a fixed unit vector in}\ \ \mathbb{S}^1.\]
 We can rewrite equation (\ref{0.3}) as follows
\begin{equation}
\label{equation-ut}\left\{\begin{array}{l}
\Delta_M u_t+\rho\left(\frac{h(x)e^{u_t(x)- G_t(x)}}
{\int_Mhe^{u_t-G_t}{d}v_g}-1\right)=0,\\ \int_M u_t {d}v_g=0,
\end{array}\right.
\end{equation}
where
\begin{align}
\label{def_of_Gt}
 & G_t (x): =4\pi   G(x,t\underline{e})+4\pi  G(x,-t\underline{e}),\ \ \mathrm{and}\\
& h(x): =h_*(x)\exp(-4\pi\sum_{i=3}^N\alpha_iG(x,q_i))\ge 0,\ \ h\in C^{3}(M), \ \ h(0)>0.\label{prop_of_h}
\end{align}  From Theorem B, we have that $u_t(x)\to w(x)+8 \pi G(x,{0})$ in $C_{\mathrm{loc}}^2(M\setminus\{{0}\})$ and $w$ satisfies
\begin{equation}
\label{thc.equationw}\left\{\begin{array}{l}
\Delta_M w+(\rho-8\pi)\left(\frac{h(x)e^{w(x) }}
{\int_Mhe^{w }{d}v_g}-1\right)=0,\\ \int_M w {d}v_g=0,\ \ \ w\in C^2(M).
\end{array}\right.
\end{equation}
We  assume that the local isothermal coordinate system satisfies
\begin{equation}
\label{def-psi0}ds^2=e^{2\varphi(x)}|dx|^2,\ 
\varphi(0)=\nabla\varphi(0)=0,
\end{equation}
that is, $e^{2\varphi}\Delta_M=\Delta$, where   $\Delta=\sum_{i=1}^2\frac{\partial^2}{\partial x_i^2}$
denotes the standard Laplacian in $\mathbb{R}^2$. Fix     a small  constant $r_0\in(0,\frac{1}{2})$.
It is well known that the conformal factor $\varphi$ is a solution of 
\begin{equation}\label{gauss}
-\Delta\varphi=e^{2\varphi}K\ \ \textrm{in}\ \ B_{r_0}(0),
\end{equation}where $K(p)$ is the Gaussian curvature at $p\in M$. \\ 
Let $\bar{\varphi}(x)$ satisfy the following local problem:
\begin{equation}\label{var}\Delta \bar{\varphi} -e^{2\varphi}\rho= 0 \ \ \textrm{in}\ \  B_{r_0} (0),\ \  \bar{\varphi}(0) =\nabla\bar{\varphi}(0) = 0.\end{equation}
We denote
\begin{equation}\label{def-psi1}\psi=2\varphi+\bar{\varphi}.
\end{equation}
In view of \eqref{def-psi0} and \eqref{var}, we note that
\begin{equation}\label{note}
  \nabla_x \psi(x)=\nabla \psi(0)+O(|x|)=O(|x|), \ \nabla_x \bar{\varphi}(x)=O(|x|),  \ \nabla_x {\varphi}(x)=O(|x|) \ \textrm{for} \ \ x\in B_{r_0}(0).
\end{equation}
By using the local coordinate, we also  set the regular part of Green function $G(x,q_i(t))$ to be
\begin{equation}\label{defr}R(x,q_i(t))=G(x,q_i(t))+\frac{1}{2\pi}\ln |x-q_i(t)|.\end{equation}
Let
\begin{equation}
\label{def-gt}
\begin{aligned}
R_t(x):&=4\pi R(x, t\underline{e})+4\pi R(x,-t\underline{e}).
\end{aligned}
\end{equation}
Therefore we can formulate the local version of (\ref{equation-ut}) around ${0}$ as follows:
\begin{equation}
\label{eq115}
\Delta\bar u_t + h_1(x)|x-t\underline{e}|^{2}|x+t\underline{e}|^{2} e^{\bar u_t(x)} = 0\ \text{in}\ B_{r_0}(0),
\end{equation}
where
\begin{equation}
\label{def-h1}\bar u_t(x)=u_t(x)-\ln \int_Mhe^{u_t-G_t}{d}v_g-\bar{\varphi}(x), \quad\ 
h_1(x)=\rho h(x)e^{\psi(x)-R_t(x)},\quad h_1(x)>0~\mathrm{in}~B_{r_0}(0).
\end{equation}

In order to study the behaviour of $\bar u_t$ near the origin, we consider the scaled sequence
\begin{equation}
\label{1-vt}
v_t(y) =\bar u_t(ty) + 6\ln \ t,\ x\in B_{\frac{r_0}{t}}(0),
\end{equation}
which satisfies:
\begin{equation}
\label{eq117}
\begin{cases}
\Delta v_t + h_t(y) e^{v_t(y)} = 0\ \ \text{in}\ B_{\frac{r_0}{t}}(0),\\
\int_{B_{\frac{r_0}{t}}(0)} h_t(y)  e^{v_t(y)}{d}y \leq C,
\end{cases}
\end{equation}
with
\begin{equation}
\label{eq118}
h_t(y) = h_1(ty)|y-\underline{e}|^{2} |y+\underline{e}|^{2}=\rho h(ty)e^{\psi(ty)-R_t(ty)}|y-\underline{e}|^{2} |y+\underline{e}|^{2}.
\end{equation}
In \cite{llty}, the following result was obtained.
\medskip

\noindent {\bf{Theorem C}}. \cite[Theorem 1.2, Section 5]{llty}  {\em
 Assume that \eqref{main assume} holds and $\rho\notin8\pi\mathbb{N}$.  Suppose that  $u_t$ be a sequence of blow up solutions of \eqref{equation-ut}. Then the scaled function $v_t$ defined by  (\ref{1-vt}) blows up at $0$.
}

\medskip

\indent Now  we are going  to give refined estimates than those provided in Theorem B and Theorem C under   the non-degeneracy assumption for (\ref{thc.equationw}).
To state our result, we fix a constant $R_0>2$, and define the following notations:
\begin{align}
\label{lam}
&\lambda_t=\max_{B_{r_0}(0)}v_t=v_t(p_t),\\
\label{1.7}
&\rho_t=\frac{\int_{B_{t {R_0}}(tp_t)}\rho  {h}e^{{u}_t-G^{(2)}_t}{d}v_g}{\int_M
{h}e^{  {u}_t-G^{(2)}_t}{d}v_g},
,\\
\label{1.8}
&C_{t}=\frac{1}{8}h_1(tp_{t})|p_{t}-\underline{e}|^{2}|p_{t}+\underline{e}|^{2},
\\
\label{3.8}
&\widetilde\phi_t(x)= u_t(x)-w(x)-\rho_tG(x,tp_t),
\end{align}
Let
$\|\t\phi_t\|_*=\|\t\phi_t\|_{C^1(M\setminus B_{2R_0t}(tp_t))}.$
Then we have the following result.
\medskip

\noindent {\bf{Theorem D}}. \cite[Theorem 1.4, Section 5]{llty}  {\em Assume that \eqref{main assume} holds and $\rho\notin8\pi\mathbb{N}$.
Let $u_t$ be the sequence of blow up solutions of (\ref{equation-ut}) and $w+8\pi G(x,{0})$ be its limit in $M\setminus\{{0}\}.$ If $w$ is a non-degenerate solution of (\ref{thc.equationw}), then
\begin{enumerate}
  \item[(i)] $
       \|\t\phi_t\|_*=O(t\ln  t),$

 \item[(ii)] \begin{equation*}\begin{aligned}
              & \lambda_{t}+2\ln  t-\ln \left(\frac{\rho}{\rho-8\pi}\int_Mhe^{w}\right)+w(tp_{t})
        +2\ln  C_{t} +8\pi R(tp_t,tp_t)\\&=O(t\ln  t),\end{aligned}
             \end{equation*}
        \item[(iii)] $\rho_{t}-8\pi=O(t^2\ln  t),$
      \item[(iv)]     $
         \left|\int_Mhe^{u_t-G_t}{d}v_g-\frac{\rho}{\rho-8\pi}\int_Mhe^{w}{d}v_g\right|=O(t),
    $
    \item[(v)] $|p_t|=O(t)$.
\end{enumerate}
}
\medskip

\indent
In order to prove Theorem D, the authors in \cite{llty} analyzed the scaled function $v_t$  with the following ingredients: Set
\begin{align}
\label{s.11}
I_{t}(y)=\ln \frac{e^{\lambda_{t}}}{(1+C_{t}e^{\lambda_{t}}|y-q_{t}|^2)^2},
\end{align}
where $q_{t}$ is chosen such that $|q_{t}-p_{t}|\ll1$ and
\begin{align*}
\nabla_yI_{t}(y)\Big|_{y= p_{t}}=-t\rho_{t}\nabla_xR(x,tp_{t})\Big|_{x=tp_{t}} -t\nabla_x w(x)\Big|_{x=tp_{t}}.
\end{align*}
By direct computation, we have
\begin{align}
\label{s.12}
|q_{t}-p_{t}|=O(te^{-\lambda_t}).
\end{align}
For $y\in B_{2r_0}(p_{t}),$ we set
\begin{align}
\label{s.13}
\eta_{t}(y)=~&v_{t}(y)-I_{t}(y)-(G_{*,t}(ty)-G_{*,t}(tp_{t})),
\end{align}
where
\begin{align}\label{s.g}G_{*,t}(x)=\rho_{t}R(x,tp_{t}) +w(x).\end{align}
It is easy to see that
\begin{align}
\label{s.14}
\eta_{t}(p_{t})=v_{t}(p_{t})-I_{t}(p_{t})=O(t^2e^{-\lambda_t}),~\nabla\eta_{t}(p_{t})=0.
\end{align}
Let
 \begin{equation}\label{lambdatp}
   \Lambda_{t,+}=\sqrt{C_{t}}e^{\frac{\lambda_{t}}{2}},\ \ \textrm{and}\ \ \Lambda_{t,-}=(\Lambda_{t,+})^{-1}=\frac{e^{-\frac{\lambda_{t}}{2}}}{\sqrt{C_{t}}}
 \end{equation} and $\widetilde\eta_{t}$ be the scaled function of $\eta_{t}$, that is
\begin{align*}
\widetilde{\eta}_{t}(z)=\eta_{t}((\Lambda_{t,-} )z+p_{t})\ \ \ \mathrm{for}\ \ \ |z|\leq 2 R_0\Lambda_{t,+}.
\end{align*}
For $\widetilde{\eta}_{t}(z)$, we have the following estimate
\medskip

\noindent {\bf{Theorem E}}. \cite[Lemma 7.1]{llty}  {\em Suppose that the assumptions of Theorem  D hold. Then for any $\e\in(0,1)$, there exists a constant $C_\e>0$, independent of $t>0$ and $z\in B_{2 R_0\Lambda_{t,+}}(0)$ such that
$$|\widetilde{\eta}_{t}(z)|\leq C_{\e}(t\|\widetilde{\phi}_t\|_*+t^2)(1+|z|)^{\e}.$$
}

\section{Uniqueness of the blow up solutions with mass concentration}\label{sec_est}
To prove Theorem \ref{thm_nonconcen}  is equivalent to prove the local uniqueness of blow up solutions of the equation \eqref{equation-ut}. To show it, we argue by contradiction and suppose that \eqref{equation-ut} has two different blow up solutions $u_t^{(1)}$ and $u_t^{(2)}$, which satisfy $u_t^{(1)}, u_t^{(2)}\to w$ in $C_{\textrm{loc}}(M\setminus\{0\})$, where $w$ is a non-degenerate solution of \eqref{thc.equationw}. We will use $p_{t}^{(i)}$, $\lambda_t^{(i)}$, $\bar{u}_t^{(i)}$, $I_t^{(i)}$, $\widetilde{\phi}_t^{(i)}$, $C_t^{(i)}$,  $q_{t}^{(i)}$, $v_t^{(i)}$, $\rho_{t}^{(i)}$, $\eta_t^{(i)}$, $\widetilde{\eta}_t^{(i)}$, $G_{*,t}^{(i)}$, $\Lambda_{t,+}^{(i)}$, $\Lambda_{t,-}^{(i)}$  to denote  $p_{t}$, $\lambda_t$, $\bar{u}_t$, $I_t$, $\widetilde{\phi}_t$, $C_t$,  $q_{t}$, $v_t$, $\rho_{t}$, $\eta_t$, $\widetilde{\eta}_t$, $G_{*,t}$, $\Lambda_{t,+}$, $\Lambda_{t,-}$  in section 2 corresponding to $u_t^{(i)}$, $i=1,2$, respectively.

From Theorem D, we have  $|p_t^{(i)}|=O(t)$ for $i=1,2$. In the following lemma, we shall improve the estimation for  $|p_{t}^{(1)}-p_{t}^{(2)}|$.
 \begin{lemma}\label{lem_diff_p_t}
  $|p_t^{(1)}-p_t^{(2)}|=O\left(t^2\ln t\right)$.
\end{lemma}\begin{proof}
Recall that $v_t^{(i)}(y) =u_t^{(i)}(ty)-\ln \int_Mhe^{u_t^{(i)}-G_t}{d}v_g-\bar{\varphi}(ty) + 6\ln \ t$ satisfies
\begin{equation}
\label{s.9}
\Delta v_t^{(i)}+h_{t}(y)e^{v_t^{(i)}(y)}=0,
\end{equation}
where $
h_{t}(y)=\rho h(ty)|y-\underline{e}|^{2}|y+\underline{e}|^{2}e^{-R_t(ty)+\psi(ty)}.
$\\
On $\partial B_{2R_0}(p_t^{(i)})$, we have
\begin{equation}
\label{s.10}
\begin{aligned}
v_t^{(i)}(y)
&=-\frac{\rho_t^{(i)}}{2\pi}\ln |y-p_t^{(i)}|+\left(6-\frac{\rho_t^{(i)}}{2\pi}\right)\ln  t +G_{*,t}^{(i)}(ty)-\ln \int_Mhe^{u_t^{(i)}-G_t}{d}v_g+\widetilde\phi_t^{(i)}(ty)-\bar{\varphi}(ty),
\end{aligned}
\end{equation}
where $G_{*,t}^{(i)}(x)=\rho_t^{(i)}R(x,tp_t^{(i)}) +w(x).$\\
For any unit vector $\xi\in\mathbb{R}^2,$ we apply the Pohozaev identity to (\ref{s.9}) by multiplying $\xi\cdot\nabla v_t^{(i)}$, and obtain
\begin{equation}
\label{s.16}
\begin{aligned}
 \sum_{i=1}^2(-1)^{i+1}\int_{B_{2R_0}(p_t^{(i)})}(\xi\cdot\nabla h_t)e^{v_t^{(i)}(y)}
 &=\sum_{i=1}^2(-1)^{i+1}\int_{\p B_{2R_0}(p_t^{(i)})}\left\{(\nu\cdot\nabla v_t^{(i)})(\xi\cdot\nabla v_t^{(i)})-\frac12(\nu\cdot\xi)|\nabla v_t^{(i)}|^2\right\}\\
&+\sum_{i=1}^2(-1)^{i+1}\int_{\p B_{2R_0}(p_t^{(i)})}(\nu\cdot\xi)h_te^{v_t^{(i)}},
\end{aligned}
\end{equation}
where $\nu$ denotes the unit normal of $\partial B_{2R_0}(p_t^{(i)})$.
From \eqref{note}, we have
\begin{equation}\label{gradpsi}
  |\nabla_y\bar{\varphi}(ty)|=t|\nabla_{ty}\bar{\varphi}(ty)|=O( t^2|y|)\ \ \textrm{for}\ \ |y|\le \frac{r_0}{t}.
\end{equation} For the right hand side of (\ref{s.16}), we can use (\ref{s.10}), Theorem D, and Theorem E  to get
\begin{equation}
\label{s.17}
\begin{aligned}
 \mbox{(RHS) of \eqref{s.16}} &=\sum_{i=1}^2(-1)^{i+1}\int_{\p B_{2R_0}(p_t^{(i)})}\Big[(\nu\cdot\nabla v_t^{(i)})(\xi\cdot\nabla v_t^{(i)})-\frac12(\nu\cdot\xi)|\nabla v_t^{(i)}|^2{d}y +O(\sum_{i=1}^2e^{-\lambda_t^{(i)}})\Big]\\
&= \sum_{i=1}^2(-1)^{i}\left[t\rho_t^{(i)}\xi\cdot\nabla_xG_{*,t}^{(i)}(x)\Big|_{x=tp_t^{(i)}}+O(t\|\widetilde\phi_t^{(i)}\|_{*}+t^2)\right]
 =O(t^2\ln t).\end{aligned}
\end{equation}
For the left hand side of (\ref{s.16}), by using  Theorem D,  we get that
\begin{equation}\begin{aligned}
\label{3.014}
(LHS) &= \sum_{i=1}^2(-1)^{i+1}\int_{B_{2R_0}(p_t^{(i)} )}\left(\xi\cdot\frac{\nabla h_t( p_t^{(i)} )}{h_t( p_t^{(i)} )}\right)h_t( y)e^{{v}_t^{(i)} (y)}{d}y\\
&+\sum_{i=1}^2(-1)^{i+1}\int_{B_{2R_0}(p_t^{(i)} )}\xi\cdot\left(\frac{\nabla h_t( y)}{h_t(y)}-\frac{\nabla h_t(p_t^{(i)} )}{h_t(p_t^{(i)} )}\right)
h_t( y)e^{{v}_t^{(i)} (y)}{d}y
\\&=\sum_{i=1}^2(-1)^{i+1}\rho_t^{(1)} \left(\xi\cdot\frac{\nabla h_t( p_t^{(i)} )}{h_t( p_t^{(i)} )}\right)
\\&+\sum_{i=1}^2(-1)^{i+1}\int_{B_{2R_0}(p_t^{(i)} )}\xi\cdot\left(\frac{\nabla h_t( y)}{h_t(y)}-\frac{\nabla h_t(p_t^{(i)} )}{h_t(p_t^{(i)} )}\right)
\\&\times\frac{h_t( y)e^{\lambda_t^{(i)}+\eta_t^{(i)}+G_{*,t}^{(i)}(ty)-G_{*,t}^{(i)}(tp_t^{(i)})}}{(1+C_t^{(i)}e^{\lambda_t^{(i)}}|y-q_t^{(i)}|^2)^2}{d}y+O(t^2\ln t).
\end{aligned}\end{equation}
By the change of variable $z= \Lambda_{t,+}^{(i)} (y-p_t^{(i)})$ and $\int_{B_{2R_0\Lambda_{t,+}^{(i)}}(0)}\frac{z_k}{(1+|z|^2)^2}{d}z=0$ for $k=1,2$, we see that
\begin{equation}\begin{aligned}
\label{3.14}
&\int_{B_{2R_0}(p_t^{(i)} )}\xi\cdot\left(\frac{\nabla h_t( y)}{h_t(y)}-\frac{\nabla h_t(p_t^{(i)} )}{h_t(p_t^{(i)} )}\right)
 \frac{h_t( y)e^{\lambda_t^{(i)}+\eta_t^{(i)}+G_{*,t}^{(i)}(ty)-G_{*,t}^{(i)}(tp_t^{(i)})}}{(1+C_t^{(i)}e^{\lambda_t^{(i)}}|y-q_t^{(i)}|^2)^2}{d}y
\\&=\int_{B_{2R_0}(p_t^{(i)} )}\xi\cdot\left(\nabla \left(\frac{\nabla h_t(p_t^{(i)} )}{h_t(p_t^{(i)} )}\right)\cdot (y-p_t^{(i)})+O(|y-p_t^{(i)}|^2)\right)
 \frac{h_t( y)e^{\lambda_t^{(i)}+\eta_t^{(i)}+G_{*,t}^{(i)}(ty)-G_{*,t}^{(i)}(tp_t^{(i)})}}{(1+C_t^{(i)}e^{\lambda_t^{(i)}}|y-q_t^{(i)}|^2)^2}{d}y
\\&=\int_{B_{2R_0\Lambda_{t,+}^{(i)}}(0)}\xi\cdot\left(\nabla \left(\frac{\nabla h_t(p_t^{(i)} )}{h_t(p_t^{(i)} )}\right)\cdot (\Lambda_{t,-}^{(i)})z+O(t^2|z|^2)\right)
 \frac{h_t( (\Lambda_{t,-}^{(i)})z+p_t^{(i)})e^{\widetilde{\eta}_t^{(i)}+G_{*,t}^{(i)}(t(\Lambda_{t,-}^{(i)})z+tp_t^{(i)})-G_{*,t}^{(i)}(tp_t^{(i)})}}{C_t^{(i)}(1+|
z+\Lambda_{t,+}^{(i)}(p_t^{(i)}-q_t^{(i)})|^2)^2}{d}z
\\&=\int_{B_{2R_0\Lambda_{t,+}^{(i)}}(0)}\xi\cdot\left(\nabla \left(\frac{\nabla h_t(p_t^{(i)} )}{h_t(p_t^{(i)} )}\right)\cdot (\Lambda_{t,-}^{(i)})z+O(t^2|z|^2)\right)
 \frac{h_t(p_t^{(i)})(1+O(t|z|)+ O(|\widetilde{\eta}_t^{(i)}|)+ O(t^2))}{C_t^{(i)}(1+|z|^2)^2}{d}z
\\&=O(t^2\ln t)\ \  \textrm{for}\ \ i=1,2,
\end{aligned}\end{equation}here we used Theorem E in the last line. \\
From \eqref{s.17}-\eqref{3.14},    we have
\begin{equation}\begin{aligned}
\label{3.141}
 \frac{\nabla h_t ( p_t^{(1)} )}{ h_t ( p_t^{(1)} )}-\frac{\nabla h_t ( p_t^{(2)} )}{ h_t ( p_t^{(2)} )} =O(t^2\ln t).
\end{aligned}\end{equation}
By using  the expression \eqref{eq118} of $h_t$,  we see that \begin{equation}\begin{aligned}&\frac{ \nabla  h_t (p_t^{(1)} ) }{ h_t (p_t^{(1)} )}-\frac{ \nabla  h_t (p_t^{(2)} ) }{ h_t (p_t^{(2)} )}\\&=\nabla (\ln|y-\underline{e}|^2|y+\underline{e}|^2)|_{y=p_t^{(1)}}-\nabla(\ln|y-\underline{e}|^2|y+\underline{e}|^2)|_{y=p_t^{(2)}} +O(t|p_t^{(1)}-p_t^{(2)}|).\label{est_ht} \end{aligned}\end{equation}Note that $|p_t^{(1)}-p_t^{(2)}|=O(t)$  from Theorem D, and $\nabla^2 (\ln|y-\underline{e}|^2|y+\underline{e}|^2)|_{y=0}$ is invertible.  So \eqref{3.141} and \eqref{est_ht} yield  that   $|p_t^{(1)}-p_t^{(2)} |=O(t^2\ln t)$, and thus we complete the proof of Lemma \ref{lem_diff_p_t}.
\end{proof}
Now we are going to estimate $\|{u}_t^{(1)}-{u}_t^{(2)}\|_{L^\infty(M)}$.
 \begin{lemma}\label{lem_diff_tilu}
  \[\|{u}_t^{(1)}-{u}_t^{(2)}\|_{L^\infty(M)}=O(t\ln t).\]
\end{lemma}\begin{proof}We note that $M\setminus   B_{2R_0t}(tp_t^{(2)})\subseteq M\setminus   B_{2R_1t}(tp_t^{(1)})$ for some $R_1>0$. For $x\in M\setminus   B_{2R_0t}(tp_t^{(1)})$, we see from Theorem D that \begin{equation}\begin{aligned}\label{0green_rep}
                   {u}_t^{(1)}(x)-{u}_t^{(2)}(x)
                   &= (w(x)+\rho_t^{(1)}G(x,tp_t^{(1)})+\widetilde{\phi}_t^{(1)}(x) )- (w(x)+\rho_t^{(2)}G(x,tp_t^{(2)})+\widetilde{\phi}_t^{(2)}(x) )
                  \\& =\rho_{t}^{(1)}G(x,tp_t^{(1)})-\rho_{t}^{(2)}G(x,tp_t^{(2)})+O(t\ln t).
                \end{aligned}
\end{equation}
Together with Theorem D and Theorem E, we have for some $\theta\in(0,1)$,
\begin{equation}\begin{aligned}\label{green_rep}
                    {u}_t^{(1)}(x)-{u}_t^{(2)}(x)
                   & =-\frac{\rho_{t}^{(1)}}{2\pi} (\ln|x-tp_t^{(1)}|-\ln |x- tp_t^{(2)}|)+O(t\ln t)
                   \\&= \frac{O(t|p_t^{(1)}-p_t^{(2)}|)}{\theta |x- tp_t^{(1)}|+(1-\theta) |x- tp_t^{(2)}|}+O(t\ln t)
                   \\&=O(|p_t^{(1)}-p_t^{(2)}|)+O(t\ln t)=O(t\ln t)\ \ \textrm{for}\ \ x\in M\setminus B_{2R_0t}(tp_t^{(1)}).
                \end{aligned}
\end{equation}

We note that $   B_{2R_0t}(tp_t^{(2)})\subseteq   B_{2R_2t}(tp_t^{(1)})$ for some $R_2>0$. For $y\in   B_{2R_0}(p_t^{(1)})$, we see that \begin{equation}\begin{aligned}\label{0green_rep_eta}
                 \eta_t^{(1)}(y)-\eta_t^{(2)}(y)
                    &=\Big({v}_t^{(1)}(y)-I_t^{(1)}(y)-(G_{*,t}^{(1)}(ty)-G_{*,t}^{(1)}(tp_t^{(1)}))\Big)
                   \\&-\Big({v}_t^{(2)}(y)-I_t^{(2)}(y)-(G_{*,t}^{(2)}(ty)-G_{*,t}^{(2)}(tp_t^{(2)}))\Big)
                  \\& =u_t^{(1)}(ty)-\ln\int_Mhe^{u_t^{(1)}-G_t}{d}v_g-\lambda_t^{(1)}+2\ln(1+C_t^{(1)}e^{\lambda_t^{(1)}}|y-q_t^{(1)}|^2)
                  \\&-\Big(u_t^{(2)}(ty)-\ln\int_Mhe^{u_t^{(2)}-G_t}{d}v_g-\lambda_t^{(2)}+2\ln(1+C_t^{(2)}e^{\lambda_t^{(2)}}|y-q_t^{(2)}|^2)\Big)
                 +O(t).
                \end{aligned}
\end{equation}
By Theorem D, we have \begin{equation}\begin{aligned}\label{diffintM}\int_Mhe^{u_t^{(1)}-G_t}{d}v_g-\int_Mhe^{u_t^{(2)}-G_t}{d}v_g &= O(t),
     \end{aligned}   \end{equation}
\begin{equation}\begin{aligned}\label{diff1}&\lambda_t^{(i)}+2\ln t+2\ln C_t^{(i)}+8\pi R(tp_t^{(i)},tp_t^{(i)})-\ln \Big(\frac{\rho}{\rho-8\pi}\int_Mhe^{w}\Big)+w(tp_t^{(i)}) = O(t\ln t),
     \end{aligned}   \end{equation}and
     \begin{equation}\begin{aligned}\label{diff2}C_t^{(1)}-C_t^{(2)}&=\frac{\rho h (tp_t^{(1)})|p_t^{(1)}-\underline{e}|^2|p_t^{(1)}+\underline{e}|^2e^{-R_t(tp_t^{(1)})+\psi(tp_t^{(1)})}}{8}
     \\&-\frac{\rho h (tp_t^{(2)})|p_t^{(2)}-\underline{e}|^2|p_t^{(2)}+\underline{e}|^2e^{-R_t(tp_t^{(2)})+\psi(tp_t^{(2)})}}{8}
     \\&=O(|p_t^{(1)}-p_t^{(2)}|)=O(t^2\ln t),
     \end{aligned}   \end{equation}which imply  \begin{equation}\label{diff_lam}\lambda_t^{(1)}-\lambda_t^{(2)}=O(t\ln t).\end{equation}
   For $y\in B_{2R_0}(p_t^{(1)})$, we want to estimate
\[    2\ln(1+C_t^{(1)}e^{\lambda_t^{(1)}}|y-q_t^{(1)}|^2)-2\ln(1+C_t^{(2)}e^{\lambda_t^{(2)}}|y-q_t^{(2)}|^2).\]
In view of \eqref{s.12} and Lemma \ref{lem_diff_p_t}, we have    \begin{equation}\label{diff_qt}
     |q_t^{(1)}-q_t^{(2)}|\le \sum_{i=1}^2|p_t^{(i)}-q_t^{(i)}|+|p_t^{(1)}-p_t^{(2)}|= O(t^2\ln t).
   \end{equation}
   Let $y=q_t^{(1)}+\Lambda_{t,-}^{(1)}z$. Then we have for $y\in B_{2R_0}(p_t^{(1)})$,
      \begin{equation}\begin{aligned}\label{diff_ln}
      &2\ln(1+C_t^{(1)}e^{\lambda_t^{(1)}}|y-q_t^{(1)}|^2)-2\ln(1+C_t^{(2)}e^{\lambda_t^{(2)}}|y-q_t^{(2)}|^2)
    \\&=2\ln (1+|z|^2) -2\ln(1+\frac{C_t^{(2)}e^{\lambda_t^{(2)}}}{C_t^{(1)}e^{\lambda_t^{(1)}}}|\Lambda_{t,+}^{(1)}(y-q_t^{(1)})+\Lambda_{t,+}^{(1)}(q_t^{(1)}-q_t^{(2)})|^2)
    \\&=2\ln(1+|z|^2)-2\ln (1+(1+O(t\ln t))|z+O(t\ln t)|^2)=O(t\ln t).
   \end{aligned}\end{equation}
By Theorem E, we have
\begin{equation}\begin{aligned}\label{drhon8pi}
   &\eta_{t}^{(1)}(y)-\eta_{t}^{(2)}(y) =
   O(t\ln t)\ \ \textrm{for}\ \ y\in B_{2R_0}(p_t^{(1)}).
\end{aligned}\end{equation}
From \eqref{0green_rep_eta}-\eqref{drhon8pi}, we have  \begin{equation}\begin{aligned}\label{diff_ut_br}
   &u_{t}^{(1)}(x)-u_{t}^{(2)}(x) =
   O(t\ln t)\ \ \ \textrm{for}\ \ \ \ x\in B_{2R_0t}(tp_t^{(1)}).
\end{aligned}\end{equation}
By \eqref{green_rep} and \eqref{diff_ut_br}, we  complete the proof of Lemma \ref{lem_diff_tilu}.
\end{proof}

Let \begin{equation}\label{def_zeta}\zeta_t(x)=\frac{u_t^{(1)}(x)-u_t^{(2)}(x)}{\|u_t^{(1)}- u_t^{(2)}\|_{L^{\infty}(M)}},\end{equation}
                      and \begin{equation}\label{zetanj}
  \widetilde{\zeta}_{t}(z)=\zeta_t(t\Lambda_{t,-}^{(1)}z+tp_{t}^{(1)})-\frac{ \int_Mhe^{u_t^{(1)}-G_t}\zeta_t{d}v_g}{\int_Mhe^{u_t^{(1)}-G_t}{d}v_g}.
\end{equation}
Now we have the following estimation for the scaled function $\widetilde{\zeta}_{t}$.
\begin{lemma}
  \label{lem_est_zetanj} There are constants $b_{0}$, $b_{1}$, and $b_{2}$ satisfying
\begin{equation*}
  \widetilde{\zeta}_{t}(z)\to \widetilde{\zeta}_0(z)=b_{0}Y_0(z)+b_{1}Y_1(z)+b_{2}Y_2(z)\quad\textrm{in}\ \ C^0_{\textrm{loc}}(\mathbb{R}^2),
\end{equation*}where
$
  Y_0(z)=\frac{1-  |z|^2}{1+ |z|^2},\ \   Y_1(z)=\frac{ z_1}{1+ |z|^2},\ \    Y_2(z)=\frac{ z_2}{1+ |z|^2}.
$
\end{lemma}    \begin{proof}
First, we see that
 \begin{equation}\begin{aligned}\label{eq_zeta}
0 &=\Delta_M \zeta_t+\frac{1 }{\|u_t^{(1)}- u_t^{(2)}\|_{L^{\infty}(M)}}\left(\frac{\rho h(x)e^{u_t^{(1)}(x)-G_t(x)}}{\int_Mhe^{u_t^{(1)}-G_t}{d}v_g} -\frac{\rho h(x)e^{u_t^{(2)}(x)-G_t(x)}}{\int_Mhe^{u_t^{(2)}-G_t}{d}v_g}\right)
\\&=\Delta_M \zeta_t+\frac{\rho h(x)e^{u_t^{(1)}(x)-G_t(x)}}{\|u_t^{(1)}- u_t^{(2)}\|_{L^{\infty}(M)}\int_Mhe^{u_t^{(1)}-G_t}{d}v_g}\left(1
-\frac{ e^{u_t^{(2)}(x)-u_t^{(1)}(x)}\int_Mhe^{u_t^{(1)}-G_t}}{\int_Mhe^{u_t^{(2)}-G_t}{d}v_g}\right)
\\&=\Delta_M \zeta_t+\frac{\rho h(x)e^{u_t^{(1)}(x)-G_t(x)}}{\|u_t^{(1)}- u_t^{(2)}\|_{L^{\infty}(M)}\int_Mhe^{u_t^{(1)}-G_t}{d}v_g}\Bigg(1
\\&-\frac{ (1+ u_t^{(2)}(x)-u_t^{(1)}(x)+O(\|u_t^{(1)}- u_t^{(2)}\|_{L^{\infty}(M)}^2))\int_Mhe^{u_t^{(1)}-G_t}{d}v_g}{\int_Mhe^{u_t^{(1)}-G_t}(1+ u_t^{(2)}(x)-u_t^{(1)}(x)+O(\|u_t^{(1)}- u_t^{(2)}\|_{L^{\infty}(M)}^2)){d}v_g}\Bigg)
\\&=\Delta_M \zeta_t+\frac{\rho h(x)e^{u_t^{(1)}(x)-G_t(x)}}{ \int_Mhe^{u_t^{(1)}-G_t}{d}v_g}\Bigg(\zeta_t
 -\frac{ \int_Mhe^{u_t^{(1)}-G_t}\zeta_t{d}v_g}{\int_Mhe^{u_t^{(1)}-G_t}{d}v_g} +O(\|u_t^{(1)}- u_t^{(2)}\|_{L^{\infty}(M)}) \Bigg).\end{aligned}
\end{equation}
By using the change of variables $y=t\Lambda_{t,-}^{(1)}z+tp_{t}^{(1)}$, \eqref{1-vt}, \eqref{s.13}, we have
 \begin{equation}\begin{aligned}\label{eq_zeta1}
 \Delta_z \widetilde{\zeta}_t(z) &=-\frac{t^6(\Lambda_{t,-}^{(1)})^2\rho h(t\Lambda_{t,-}^{(1)}z+tp_{t}^{(1)})e^{u_t^{(1)}(t\Lambda_{t,-}^{(1)}z+tp_{t}^{(1)})-R_t(t\Lambda_{t,-}^{(1)}z+tp_{t}^{(1)})}}{ \int_Mhe^{u_t^{(1)}-G_t}{d}v_g}\\&\times \left|\Lambda_{t,-}^{(1)}z+p_{t}^{(1)}-\underline{e}\right|^2\left|\Lambda_{t,-}^{(1)}z+p_{t}^{(1)}+\underline{e}\right|^2 (\widetilde{\zeta}_t(z) +O(\|u_t^{(1)}- u_t^{(2)}\|_{L^{\infty}(M)})  )
 \\&=-(\Lambda_{t,-}^{(1)})^2
 h_1(t\Lambda_{t,-}^{(1)}z+tp_{t}^{(1)})e^{{v}_t^{(1)}(\Lambda_{t,-}^{(1)}z+p_{t}^{(1)})}  \\&\times \left|\Lambda_{t,-}^{(1)}z+p_{t}^{(1)}-\underline{e}\right|^2\left|\Lambda_{t,-}^{(1)}z+p_{t}^{(1)}+\underline{e}\right|^2 (\widetilde{\zeta}_t(z) +O(\|u_t^{(1)}- u_t^{(2)}\|_{L^{\infty}(M)})  )
 \\&=-\left(\frac{1}{C_t^{(1)}}\right)
  \frac{h_1(t\Lambda_{t,-}^{(1)}z+tp_{t}^{(1)})e^{G_{*,t}^{(1)}(t\Lambda_{t,-}^{(1)}z+tp_{t}^{(1)})-G_{*,t}^{(1)}(tp_t^{(1)})+\widetilde\eta_t^{(1)}
  (z)}}{(1+|z+\Lambda_{t,+}^{(1)}(p_{t}^{(1)}-q_t^{(1)})|^2)^2} \\&\times \left|\Lambda_{t,-}^{(1)}z+p_{t}^{(1)}-\underline{e}\right|^2\left|\Lambda_{t,-}^{(1)}z+p_{t}^{(1)}+\underline{e}\right|^2 (\widetilde{\zeta}_t(z) +O(\|u_t^{(1)}- u_t^{(2)}\|_{L^{\infty}(M)})  )
  \\&=\frac{-8h_1(t\Lambda_{t,-}^{(1)}z+tp_{t}^{(1)})}{h_1(tp_t^{(1)}) }  \frac{ \left|\Lambda_{t,-}^{(1)}z+p_{t}^{(1)}-\underline{e}\right|^2\left|\Lambda_{t,-}^{(1)}z+p_{t}^{(1)}+\underline{e}\right|^2}{|p_t^{(1)}-\underline{e}|^2|p_t^{(1)}+\underline{e}|^2} \\&\times \frac{  (\widetilde{\zeta}_t(z) +O(\|u_t^{(1)}- u_t^{(2)}\|_{L^{\infty}(M)})  )(1+O(|\widetilde\eta_t(z)|)+O(t^2|z|))}{(1+|z+\Lambda_{t,+}^{(1)}(p_{t}^{(1)}-q_t^{(1)})|^2)^2}.\end{aligned}
\end{equation}        Together with \eqref{s.12}, Lemma \ref{lem_diff_tilu}, and Theorem E,  we have for  $z\in B_{ 2\Lambda_{t,+}R_0}(0)$,
                   \begin{equation}\begin{aligned}\label{eq_zeta_final}
 &\Delta_z \widetilde{\zeta}_t(z)+\frac{8 \widetilde{\zeta}_t(z)}{(1+|z|^2)^2 } =-\frac{8\widetilde{\zeta}_t \nabla\ln H_t(p_{t}^{(1)})\cdot (\Lambda_{t,-}^{(1)} z)+O(t\ln t)
 +O(t^2|z|^2)}{(1+|z|^2)^2 },\end{aligned}
\end{equation}where \begin{equation}\label{htdef}
                     H_t(x)=h_1(tx)|x-\underline{e}|^2|x+\underline{e}|^2.
                   \end{equation}
Since $\widetilde{\zeta}_t$ is uniformly bounded, there is a function $\t\zeta_0$ such that $\widetilde{\zeta}_t\to\widetilde{\zeta}_0$ in $C_{\textrm{loc}}(\mathbb{R}^2)$, where
\begin{equation}
\label{zeta0eq}
\left\{\begin{array}{l}
 \Delta \widetilde{\zeta}_0+\frac{8\widetilde{\zeta}_0}{(1+|z|^2)^2}=0\ \ \textrm{in}\ \ \mathbb{R}^2,\\
\|\widetilde{\zeta}_0\|_{L^\infty(\mathbb{R}^2)}\le c\ \ \textrm{for some constant}\ \ c>0.
\end{array}\right.
\end{equation}
 By \cite[Proposition 1]{bp}, we see that    $\widetilde{\zeta}_0(z)=\sum_{i=0}^2b_{i}Y_i(z)$ for some constants $b_i\in \mathbb{R}^2$, $i=0,1,2$. This completes the proof of Lemma \ref{lem_est_zetanj}.\end{proof}
In the following lemma, we observe the behavior of $\zeta_t$ in $M\setminus\{0\}$.
\begin{lemma}\label{lem_equalb0}
 (i)   $\zeta_t\to 0$ in $ C^{0}_{\textrm{loc}}(M\setminus\{0\})$,

 (ii) $\lim_{t\to0}\left(\int_M he^{u_t^{(1)}-G_t}\zeta_t{d}v_g\right)=0.$
 \end{lemma}
\begin{proof}
We recall from \eqref{eq_zeta} that in $M$, \[\Delta_M \zeta_t+\frac{\rho h(x)e^{u_t^{(1)}-G_t}}{\int_M he^{u_t^{(1)}-G_t}{d}v_g}\left(\zeta_t-\frac{ \int_M he^{u_t^{(1)}-G_t}\zeta_t{d}v_g}{\int_M he^{u_t^{(1)}-G_t}{d}v_g}+O(\|u_t^{(1)}-u_t^{(2)}\|_{L^\infty(M)}\right)=0.\]
 Since  $\|\zeta_t\|_{L^\infty(M)}\le 1$, we see that there is a function $\zeta_*$ satisfying  \begin{equation}\label{convzeta}
                                                                                              \zeta_t\to\zeta_*\ \ \textrm{in}\ \ C_{\textrm{loc}}(M\setminus\{0\}).
                                                                                                \end{equation}
                                                                                                From Theorem D,  we have
 \begin{equation}\label{inthlim}
   \lim_{t\to0}\int_M he^{u_t^{(1)}-G_t}{d}v_g=\frac{\rho}{\rho-8\pi}\int_M he^{w}{d}v_g.
 \end{equation}
 For any small fixed $r\in(0,1)$, we see from Theorem D that
 \begin{equation}\label{withzetaint}
   \begin{aligned} \int_M he^{u_t^{(1)}-G_t}\zeta_t{d}v_g
   &=\left[\int_{M\setminus B_r(0)}+\int_{B_r(0)\setminus B_{2R_0t}(tp_t^{(1)})}\right] he^{u_t^{(1)}-G_t}\zeta_t{d}v_g
   \\&+\int_{B_{2R_0t}(tp_t^{(1)})}he^{u_t^{(1)}-G_t}\left(\zeta_t- \frac{\int_{M}he^{u_t^{(1)}-G_t}\zeta_t{d}v_g}{\int_{M}he^{u_t^{(1)}-G_t} {d}v_g}\right){d}v_g
   \\&+  \int_{B_{2R_0t}(tp_t^{(1)})}he^{u_t^{(1)}-G_t}{d}v_g\frac{\int_{M}he^{u_t^{(1)}-G_t}\zeta_t{d}v_g}{\int_{M}he^{u_t^{(1)}-G_t} {d}v_g}
   \\&=\int_Mhe^{w}\zeta_*{d}v_g
   +\frac{1}{\rho}\left(\int_Mhe^{u_t^{(1)}-G_t}{d}v_g\right)\int_{B_{2R_0t}(tp_t^{(1)})}(-\Delta_M\zeta_t){d}v_g\\&+\frac{8\pi}{\rho}\int_M he^{u_t^{(1)}-G_t}\zeta_t{d}v_g+o(1)+O(r^2)+O(\|u_t^{(1)}-u_t^{(2)}\|_{L^\infty(M)}).\end{aligned}
 \end{equation}
By using the change of variable $x=t(\Lambda_{t,-}^{(1)}z+p_t^{(1)})$, we note that as $t\to0$,
\begin{equation}\begin{aligned}\label{delta_zeta}
  \int_{B_{2R_0t}(tp_t^{(1)})}-\Delta_x\zeta_t(x){d}x &=\int_{B_{2R_0\Lambda_{t,+}^{(1)}}(0)}-\Delta_z\widetilde{\zeta}_t(z){d}z
   =\int_{B_{2R_0\Lambda_{t,+}^{(1)}}(0)}\frac{8\widetilde{\zeta}_t(z)+O(t|z|)+o(1)}{(1+|z|^2)^2} {d}z=o(1),
\end{aligned}\end{equation}
since $\widetilde{\zeta}_t\to\sum_{j=0}^2b_jY_j$ in $C_{\textrm{loc}}(\mathbb{R}^2)$ and $\int_{\mathbb{R}^2}\frac{Y_i}{(1+|z|^2)^2}=0$ for $i=0,1,2$.
So we obtain from \eqref{withzetaint} and \eqref{delta_zeta} that
\begin{equation*}
  \left(1-\frac{8\pi}{\rho}\right)\int_Mhe^{u_t^{(1)}-G_t}\zeta_t{d}v_g=\int_Mhe^w\zeta_*{d}v_g+o(1),
\end{equation*}
which implies
\begin{equation}\label{1minus}
\int_Mhe^{u_t^{(1)}-G_t}\zeta_t{d}v_g= \left(\frac{\rho}{\rho-8\pi}\right)\int_Mhe^w\zeta_*{d}v_g+o(1).
\end{equation}
Then we have \begin{equation}\label{zetasta}
               \Delta_M \zeta_*+\frac{(\rho-8\pi)he^{w}}{\int_Mhe^{w}{d}v_g}\left(\zeta_*-\frac{\int_Mhe^{w}\zeta_*{d}v_g}{\int_Mhe^w{d}v_g}\right)=0\ \ \textrm{in}\ \ M\setminus\{0\}.             \end{equation}
               Since $\|\zeta_*\|_{L^\infty(M)}\le 1$, the above equation \eqref{zetasta} holds in $M$.
               Moreover, we note that  \begin{equation}\label{int0}\int_M\zeta_t{d}v_g=\frac{\int_M(u_t^{(1)}-u_t^{(2)}){d}v_g}{\|u_t^{(1)}-u_t^{(2)}\|_{L^\infty(M)}}=0,\end{equation} and thus $\int_M\zeta_*{d}v_g=0$.
               Together with non-degeneracy condition for $w$, we obtain $\zeta_*\equiv0$. In view of \eqref{convzeta} and \eqref{1minus}, we complete the proof of Lemma \ref{lem_equalb0}.
 \end{proof}To connect the behavior of $\zeta_t$ in $M\setminus\{0\}$ and in a small neighborhood of $0$, we need the following result.
\begin{lemma}\label{ref}\cite{lly}
  (i) If $\frac{tR_0}{2}\le |x_2-tp_t^{(1)}|\le |x_1-tp_t^{(1)}|\le r_0$, then
  \begin{equation}\begin{aligned}\label{diffzeta}
    \zeta_t(x_1)-\zeta_t(x_2)&=O\left(\ln\frac{|x_1-tp_t^{(1)}|}{|x_2-tp_t^{(1)}|}\int_{B_{2R_0\Lambda_{t,+}^{(1)}}(0)}\Delta\widetilde{\zeta}_t{d}z\right) +O(|x_1-tp_t^{(1)}|\ln|x_1-tp_t^{(1)}|)
    +O(t^{\frac{\alpha}{2}}\ln t),
         \end{aligned} \end{equation}

(ii) If $t^2 R_0\le |x-tp_t^{(1)}|\le \frac{tR_0}{2}$, then
  \begin{equation}\label{diffzeta2}\begin{aligned}
     \zeta_t(x)-\zeta_t(tp_t^{(1)}) &=O\left(\int_{B_{2R_0\Lambda_{t,+}^{(1)}}(0)}(\ln |z|)\Delta\widetilde{\zeta}_t{d}z\right)+O\left(\ln\frac{|x-tp_t^{(1)}|}{t^2}\int_{B_{2R_0\Lambda_{t,+}^{(1)}}(0)}\Delta\widetilde{\zeta}_t{d}z\right)
    \\&+O\left(\left(\frac{|x-tp_t^{(1)}|}{t^2}\right)^{-\frac{\alpha}{2}}\ln\left(\frac{|x-tp_t^{(1)}|}{t^2}\right)\right)
    +O(t\ln t).\end{aligned}
          \end{equation}
\end{lemma}\begin{proof}For any function $g$ satisfying $g(z)(1+|z|)^{1+\frac{\alpha}{2}}\in L^2(\RN)$, we recall  the following estimation (see \cite{cfl}): there is a constant $c >0$, independent of $x\in\RN\setminus B_2(0)$ and $g$, such that
\[\left|\intr (\ln|x-z|-\ln|x|)g(z)\mathrm{d}z\right|\le c  |x|^{-\frac{\alpha}{2}}(\ln|x|+1)\|g(z)(1+|z|)^{1+\frac{\alpha}{2}}\|_{L^2(\RN)}.\]
Together with  the Green representation formula,             Lemma \ref{ref} can be obtained.
             See \cite{lly} for the detail.
           \end{proof}
  Let $\chi_t$ be a cut-off function satisfying $0\le \chi_t\le 1$, $|\nabla\chi_t|=O(t)$, $|\nabla^2\chi_t|=O(t^2)$, and \begin{equation}
\label{chit}
\chi_t(z)=\chi_t(|z|)=\left\{\begin{array}{l}
1\ \ \textrm{if}\ \ |z|\le R_0\Lambda_{t,+}^{(1)},\\
0\ \ \textrm{if}\ \ |z|\ge 2R_0\Lambda_{t,+}^{(1)}.
\end{array}\right.
\end{equation}Then we have the following result.
\begin{lemma}
  \label{lem_intzetathree}
  \begin{enumerate}
  \item[(i)] $\int_{B_{2R_0\Lambda_{t,+}^{(1)}}(0)}\frac{\widetilde{\zeta}_t(z)\chi_t(z)}{(1+|z|^2)^2}{d}z=O(t\ln t),$
  \item[(ii)] $\int_{B_{2R_0\Lambda_{t,+}^{(1)}}(0)}\Delta \widetilde{\zeta}_t{d}z=\int_{B_{2R_0t(tp_t^{(1)})}}\Delta \zeta_t{d}x=O(t\ln t)$,
  \item[(iii)] $\lim_{t\to0}\|\zeta_t\|_{L^\infty(M\setminus B_{\frac{tR_0}{2}}(tp_t^{(1)}))}=0 $,
  \item[(iv)]  $\lim_{t\to0}\int_{B_{2R_0\Lambda_{t,+}^{(1)}}(0)}\frac{\widetilde{\zeta}_t(z)Y_0(z)\chi_t(z)}{(1+|z|^2)^2}{d}z=o(1)$,
  \item[(v)]  $b_0=0$.
\end{enumerate}
\end{lemma}
\begin{proof}
  (i) We note that $\eta_1(z)=-\frac{2}{(1+|z|^2)}$ satisfies
  \begin{equation}\label{eta1}
    \Delta\eta_1+\frac{8\eta_1}{(1+|z|^2)^2}=-\frac{8}{(1+|z|^2)^2}\ \ \textrm{in}\ \ \mathbb{R}^2.
  \end{equation}
From \eqref{eq_zeta_final}, we recall the following equation:
                   \begin{equation*}\begin{aligned}
 &\Delta_z \widetilde{\zeta}_t(z)+\frac{8 \widetilde{\zeta}_t(z)}{(1+|z|^2)^2 } =-\frac{8(\Lambda_{t,-}^{(1)})\widetilde{\zeta}_t\nabla\ln H_t(p_{t}^{(1)})\cdot z+O(t\ln t)
 +O(t^2|z|^2)}{(1+|z|^2)^2 }.\end{aligned}
\end{equation*}
Multiplying both sides of \eqref{eq_zeta_final} by $\eta_1{\chi_t}$  and using the integration by parts, we have
\begin{equation}\label{multieta1}
\begin{aligned}  0&=\int_{B_{2R_0\Lambda_{t,+}^{(1)}}(0)}\widetilde{\zeta}_t\left(\Delta(\eta_1{\chi_t})+\frac{8\eta_1{\chi_t}}{(1+|z|^2)^2}\right){d}z+O(t\ln t)
\\&=\int_{B_{2R_0\Lambda_{t,+}^{(1)}}(0)}\widetilde{\zeta}_t\left[\left(\Delta \eta_1+\frac{8\eta_1}{(1+|z|^2)^2}\right){\chi_t}+2\nabla\eta_1\cdot\nabla{\chi_t}+\eta_1\Delta{\chi_t}\right]{d}z +O(t\ln t)
.
\end{aligned}\end{equation}
Together with \eqref{eta1}, we obtain
\begin{equation}\label{remultieta1}
\int_{B_{2R_0\Lambda_{t,+}^{(1)}}(0)}\frac{8\widetilde{\zeta}_t{\chi_t}}{(1+|z|^2)^2}{d}z=O(t\ln t).
\end{equation}

(ii) By integrating \eqref{eq_zeta_final} and using \eqref{remultieta1}, we have Lemma \ref{lem_intzetathree}-(ii).

(iii) By Lemma \ref{ref}-(i) and Lemma \ref{lem_equalb0}-(i), we see that  if $\frac{tR_0}{2}\le |x-tp_t^{(1)}|\le r_0$ and $|x'-tp_t^{(1)}|=r$, then
\[\zeta_t(x)=\zeta_t(x')+O(\ln t\int_{B_{2R_0\Lambda_{t,+}^{(1)}}(0)}\Delta\widetilde{\zeta}_t {d}z)+O(r\ln r)+o(1),\]for any small $r>0$.
Together with Lemma \ref{lem_intzetathree}-(ii), we can get that Lemma \ref{lem_intzetathree}-(iii).

(iv)   We note that $\eta_2(z)=\frac{4}{3}\ln(1+|z|^2)\left(\frac{1-|z|^2}{1+|z|^2}\right)+\frac{8}{3(1+|z|^2)}$ satisfies
  \begin{equation}\label{eta2}
    \Delta\eta_2+\frac{8\eta_2}{(1+|z|^2)^2}=\frac{16Y_0(z)}{(1+|z|^2)^2}\ \ \textrm{in}\ \ \mathbb{R}^2.
  \end{equation}
Multiplying both sides of \eqref{eq_zeta_final} by $\eta_2{\chi_t}$   and using the integration by parts, we have
\begin{equation}\label{multieta2}
\begin{aligned}  0& =\int_{B_{2R_0\Lambda_{t,+}^{(1)}}(0)}\widetilde{\zeta}_t\left[\left(\Delta \eta_2+\frac{8\eta_2}{(1+|z|^2)^2}\right){\chi_t}+2\nabla\eta_2\cdot\nabla{\chi_t}+\eta_2\Delta{\chi_t}\right]{d}z +O(t\ln t)
.
\end{aligned}\end{equation}
Fix a point $e_t\in \partial B_{R_0\Lambda_{t,+}^{(1)}}(0)$. Then \eqref{multieta2} implies
\begin{equation}\label{remultieta2}
\begin{aligned}
 \int_{B_{2R_0\Lambda_{t,+}^{(1)}}(0)}\frac{16Y_0\widetilde{\zeta}_t\chi_t}{(1+|z|^2)^2}{d}z
 &=
-\int_{B_{2R_0\Lambda_{t,+}^{(1)}}(0)}(\widetilde{\zeta}_t(z)-\widetilde{\zeta}_t(e_t)))(2\nabla\eta_2\cdot\nabla{\chi_t}+\eta_2\Delta{\chi_t})
{d}z
\\&-\widetilde{\zeta}_t(e_t)\int_{B_{2R_0\Lambda_{t,+}^{(1)}}(0)}(2\nabla\eta_2\cdot\nabla{\chi_t}+\eta_2\Delta{\chi_t})
{d}z+O(t\ln t).\end{aligned}
\end{equation}
Together with Lemma \ref{ref}-(i) and Lemma \ref{lem_intzetathree}-(ii), we have
\begin{equation}\label{reeta2}\begin{aligned}
 \int_{B_{2R_0\Lambda_{t,+}^{(1)}}(0)}\frac{16Y_0\widetilde{\zeta}_t\chi_t}{(1+|z|^2)^2}{d}z &=
\int_{R_0\Lambda_{t,+}^{(1)}\le |z|\le 2R_0\Lambda_{t,+}^{(1)}}O(t^{\frac{\alpha}{2}}\ln t)(\frac{t}{|z|}+|\ln |z||t^2){d}z +O(\widetilde{\zeta}_t(e_t))+O(t\ln t),
\end{aligned}\end{equation}here we used $\int_{B_{2R_0\Lambda_{t,+}^{(1)}}(0)} \eta_2\Delta{\chi_t}
{d}z=-\int_{B_{2R_0\Lambda_{t,+}^{(1)}}(0)} \nabla\eta_2\cdot\nabla{\chi_t} {d}z$. \\
By applying  Lemma \ref{lem_equalb0}-(ii) and Lemma \ref{lem_intzetathree}-(iii), we obtain $\lim_{t\to0}\widetilde{\zeta}_t(e_t)=0$, and thus
\begin{equation}\label{reeta3}\begin{aligned}
&\int_{B_{2R_0\Lambda_{t,+}^{(1)}}(0)}\frac{16Y_0\widetilde{\zeta}_t\chi_t}{(1+|z|^2)^2}{d}z=o(1)\ \ \textrm{as}\ \ t\to0.
\end{aligned}\end{equation} So we obtain Lemma \ref{lem_intzetathree}-(iv).

(v) By Lemma \ref{lem_est_zetanj} and Lemma \ref{lem_intzetathree}-(iv), we have
\begin{equation}\label{b0}
  b_0\equiv0.
\end{equation}
So we complete the proof of Lemma \ref{lem_intzetathree}.
\end{proof}
Let \begin{equation}\label{tildeut}
      \widetilde{u}_t^{(i)}=u_t^{(i)}-\ln\int_Mhe^{u_t^{(i)}-G_t}{d}v_g\ \ \textrm{for}\ \ i=1,2.
    \end{equation}
We note that
\begin{equation}\label{equt1diffut2}
  \begin{aligned}
   \widetilde{u}_t^{(1)}-\widetilde{u}_t^{(2)} &=u_t^{(1)}-u_t^{(2)}-\ln\int_Mhe^{u_t^{(1)}-G_t}{d}v_g+\ln\int_Mhe^{u_t^{(2)}-G_t}{d}v_g
  \\&=u_t^{(1)}-u_t^{(2)}-\ln\int_Mhe^{u_t^{(1)}-G_t}{d}v_g +\ln\int_Mhe^{u_t^{(1)}-G_t}(1+u_t^{(2)}-u_t^{(1)}+O(|u_t^{(1)}-u_t^{(2)}|^2){d}v_g
  \\&= u_t^{(1)}-u_t^{(2)}-\frac{\int_Mhe^{u_t^{(1)}-G_t}(u_t^{(1)}-u_t^{(2)}){d}v_g}{\int_Mhe^{u_t^{(1)}-G_t} {d}v_g}+O(\|u_t^{(1)}-u_t^{(2)}\|^2_{L^\infty(M)}).\end{aligned}
\end{equation}
Let
\begin{equation}\label{at}
  A_t:=\int_{B_{2R_0 t}(tp_t^{(1)})}\frac{\rho h e^{-G_t}(e^{\widetilde{u}_t^{(1)}}-e^{\widetilde{u}_t^{(2)}})}{\|u_t^{(1)}-u_t^{(2)}\|_{L^\infty(M)}}{d}x=\int_{B_{2R_0 t}(tp_t^{(1)})}-\Delta \zeta_t{d}x.
\end{equation}
Without loss of generality, from now on, we  assume that
\begin{equation}\label{wemayassume}
  \nabla_x(8\pi R(x,0)+w(x))\Big|_{x=0}=0.
\end{equation}
Indeed, we can change the regular part of $G(x,0)$ locally such that\[8\pi R_{\textrm{new}}(x,0)=8\pi R_{\textrm{old}}(x,0)-\nabla_x(8\pi R_{\textrm{old}}(x,0)+w(x))\Big|_{x=0}\cdot x.\]
Now we shall improve Lemma \ref{lem_intzetathree}-(ii) by applying the arguments in \cite{ly2}.
\begin{lemma}
  \label{lem_estat}
  \begin{equation*}
    A_t=\int_{B_{2R_0t}(tp_t^{(1)})}-\Delta\zeta_t{d}x=O(t).
  \end{equation*}
\end{lemma}
\begin{proof}
Recall that \begin{equation}\begin{aligned}\label{hatvt}
      {v}_t^{(i)}(y)&=\widetilde{u}_t^{(i)}(ty)+6\ln t -\bar{\varphi}(ty)
      \\&=\eta_t^{(i)}(y)+I_t^{(i)}(y)+G_{*,t}^{(i)}(ty)-G_{*,t}^{(i)}(tp_t^{(i)})\ \ \textrm{for}\ \ i=1,2.\end{aligned}
    \end{equation}Set \begin{equation}\label{widetildv}
                        \widetilde{v}_t^{(i)}(z)=v_t^{(i)}(\Lambda_{t,-}^{(1)}z+p_t^{(1)}) \ \ \textrm{for}\ \ i=1,2.
                      \end{equation}Then
    \begin{equation}\begin{aligned}\label{tilvt}
      \widetilde{v}_t^{(i)}(z)&= \eta_t^{(i)}(\Lambda_{t,-}^{(1)}z+p_t^{(1)})
      +\ln \frac{e^{\lambda_t^{(i)}}}{(1+(\Lambda_{t,+}^{(i)})^2|\Lambda_{t,-}^{(1)}z+p_t^{(1)}
      -q_t^{(i)})|^2)^2} +G_{*,t}^{(i)}(t\Lambda_{t,-}^{(1)}z+tp_t^{(1)})-G_{*,t}^{(i)}(tp_t^{(i)}).\end{aligned}
    \end{equation}
    We also see from \eqref{equt1diffut2} that
    \begin{equation}\label{difftilv}\begin{aligned}
      \frac{\widetilde{v}_t^{(1)}(z)-\widetilde{v}_t^{(2)}(z)}{\|u_t^{(1)}-u_t^{(2)}\|_{L\infty(M)}}&=\frac{\widetilde{u}_t^{(1)}(t\Lambda_{t,-}^{(1)}z+tp_t^{(1)})
      -\widetilde{u}_t^{(2)}(t\Lambda_{t,-}^{(1)}z+tp_t^{(1)})}{\|u_t^{(1)}-u_t^{(2)}\|_{L\infty(M)}}\\&=\widetilde{\zeta}_t(z)+O(\|u_t^{(1)}-u_t^{(2)}\|_{L^\infty(M)}),
    \end{aligned}\end{equation}
    which implies
    \begin{equation}\label{difftilv0}\begin{aligned}
      \frac{1-e^{\widetilde{v}_t^{(2)}-\widetilde{v}_t^{(1)}}}{\|u_t^{(1)}-u_t^{(2)}\|_{L\infty(M)}}&=\frac{\widetilde{v}_t^{(1)}(z)
      -\widetilde{v}_t^{(2)}(z)+O(|\widetilde{v}_t^{(1)}-\widetilde{v}_t^{(2)}|^2)}{\|u_t^{(1)}-u_t^{(2)}\|_{L\infty(M)}} =\widetilde{\zeta}_t(z)+O(\|u_t^{(1)}-u_t^{(2)}\|_{L^\infty(M)}).
    \end{aligned}\end{equation}
We have
\begin{equation}\label{eqtilv}
  \begin{aligned}
  \Delta_z\widetilde{v}_t^{(i)}(z)
  +(\Lambda_{t,-}^{(1)})^2h_t(\Lambda_{t,-}^{(1)}z+p_t^{(1)})e^{\widetilde{v}_t^{(i)}(z)}=0,
  \end{aligned}
\end{equation}where $h_t(y)=\rho h(ty)|y-\underline{e}|^2|y+\underline{e}|^2e^{-R_t(ty)+\psi(ty)}.$
We see that
\begin{equation}\label{eqtilv1}
  \begin{aligned}
  &(\Delta(\widetilde{v}_t^{(1)}-\widetilde{v}_t^{(2)}))(\nabla(\widetilde{v}_t^{(1)}+\widetilde{v}_t^{(2)})\cdot z)+
   (\Delta(\widetilde{v}_t^{(1)}+\widetilde{v}_t^{(2)}))(\nabla(\widetilde{v}_t^{(1)}-\widetilde{v}_t^{(2)})\cdot z)
 \\&=\mbox{div}\Big\{(\nabla(\widetilde{v}_t^{(1)}-\widetilde{v}_t^{(2)}))(\nabla(\widetilde{v}_t^{(1)}+\widetilde{v}_t^{(2)})\cdot z)\\&+(\nabla(\widetilde{v}_t^{(1)}+\widetilde{v}_t^{(2)}))(\nabla(\widetilde{v}_t^{(1)}-\widetilde{v}_t^{(2)})\cdot z)-\nabla(\widetilde{v}_t^{(1)}-\widetilde{v}_t^{(2)})\cdot\nabla(\widetilde{v}_t^{(1)}+\widetilde{v}_t^{(2)}) z\Big\},
  \end{aligned}
\end{equation}and
\begin{equation}\label{eqtilv2}
  \begin{aligned}
  &(\Delta(\widetilde{v}_t^{(1)}-\widetilde{v}_t^{(2)}))(\nabla(\widetilde{v}_t^{(1)}+\widetilde{v}_t^{(2)})\cdot z)+
   (\Delta(\widetilde{v}_t^{(1)}+\widetilde{v}_t^{(2)}))(\nabla(\widetilde{v}_t^{(1)}-\widetilde{v}_t^{(2)})\cdot z)
 \\&= -(\Lambda_{t,-}^{(1)})^2h_t(\Lambda_{t,-}^{(1)}z+p_t^{(1)})(e^{\widetilde{v}_t^{(1)}(z)}-e^{\widetilde{v}_t^{(2)}(z)})(\nabla(\widetilde{v}_t^{(1)}+\widetilde{v}_t^{(2)})\cdot z)
 \\&  -(\Lambda_{t,-}^{(1)})^2h_t(\Lambda_{t,-}^{(1)}z+p_t^{(1)})(e^{\widetilde{v}_t^{(1)}(z)}+e^{\widetilde{v}_t^{(2)}(z)})(\nabla(\widetilde{v}_t^{(1)}-\widetilde{v}_t^{(2)})\cdot z)
\\&=-\mbox{div}\left(2(\Lambda_{t,-}^{(1)})^2h_t(\Lambda_{t,-}^{(1)}z+p_t^{(1)})(e^{\widetilde{v}_t^{(1)}(z)}-e^{\widetilde{v}_t^{(2)}(z)})z\right)
\\&+4(\Lambda_{t,-}^{(1)})^2h_t(\Lambda_{t,-}^{(1)}z+p_t^{(1)})(e^{\widetilde{v}_t^{(1)}(z)}-e^{\widetilde{v}_t^{(2)}(z)})
\\&+2(\Lambda_{t,-}^{(1)})^2h_t(\Lambda_{t,-}^{(1)}z+p_t^{(1)})(e^{\widetilde{v}_t^{(1)}(z)}-e^{\widetilde{v}_t^{(2)}(z)})\left(\nabla_z\ln h_t(\Lambda_{t,-}^{(1)}z+p_t^{(1)})\cdot z\right).\end{aligned}
\end{equation}
Therefore, we obtain for any $r>0$,
\begin{equation}\label{eqtilv3}
  \begin{aligned} &\frac{1}{2}\int_{\partial B_r(0)}\nabla (\widetilde{v}_t^{(1)}-\widetilde{v}_t^{(2)})\cdot \nabla (\widetilde{v}_t^{(1)}+\widetilde{v}_t^{(2)})|z|{d}\sigma
   -\int_{\partial B_r(0)}\frac{(\nabla (\widetilde{v}_t^{(1)}-\widetilde{v}_t^{(2)})\cdot z)(\nabla (\widetilde{v}_t^{(1)}+\widetilde{v}_t^{(2)})\cdot z)}{|z|}{d}\sigma
\\&=\int_{\partial B_r(0)}(\Lambda_{t,-}^{(1)})^2h_t(\Lambda_{t,-}^{(1)}z+p_t^{(1)}) e^{\widetilde{v}_t^{(1)}(z)}(1-e^{\widetilde{v}_t^{(2)}(z)-\widetilde{v}_t^{(1)}(z)})|z|{d}\sigma\\&-\int_{ B_r(0)}(\Lambda_{t,-}^{(1)})^2h_t(\Lambda_{t,-}^{(1)}z+p_t^{(1)}) e^{\widetilde{v}_t^{(1)}(z)}(1-e^{\widetilde{v}_t^{(2)}(z)-\widetilde{v}_t^{(1)}(z)}) \left(2+\nabla_z\ln h_t(\Lambda_{t,-}^{(1)}z+p_t^{(1)})\cdot z\right){d}z.\end{aligned}
\end{equation}
Let $  2R_0\Lambda_{t,+}^{(1)}\le |z|\le \frac{r_0}{t}\Lambda_{t,+}^{(1)}$. By \eqref{note} and Theorem D, we have \[\nabla_z\bar{\varphi}(t\Lambda_{t,-}^{(1)}z+tp_t^{(1)})=t^2O(t\Lambda_{t,-}^{(1)}z+tp_t^{(1)})=O(t^4(|z|+1)),\]and
\begin{equation}\label{gradtilv}
  \begin{aligned}  \nabla_z\widetilde{v}_t^{(i)}(z) &=\nabla_z\left(\widetilde{u}_t^{(i)}(t\Lambda_{t,-}^{(1)}z+tp_t^{(1)}) -
  \bar{\varphi}(t \Lambda_{t,-}^{(1)}z+t p_t^{(1)})\right)
  \\&=\nabla_z\widetilde{\phi}_t^{(i)}(t\Lambda_{t,-}^{(1)}z+tp_t^{(1)}) +\nabla_z\left(\rho_t^{(i)}G(t\Lambda_{t,-}^{(1)}z+tp_t^{(1)},tp_t^{(i)})+w(t\Lambda_{t,-}^{(1)}z+tp_t^{(1)})\right)
  +O(t^4|z|)
  \\&=\nabla_z\widetilde{\phi}_t^{(i)}(t\Lambda_{t,-}^{(1)}z+tp_t^{(1)})-\frac{\rho_t^{(i)}t\Lambda_{t,-}^{(1)}}{2\pi}
  \frac{(t\Lambda_{t,-}^{(1)}z+tp_t^{(1)}-tp_t^{(i)})}{|t\Lambda_{t,-}^{(1)}z+tp_t^{(1)}-tp_t^{(i)}|^2}
  \\&+\nabla_z\left(\rho_t^{(i)}R(t\Lambda_{t,-}^{(1)}z+tp_t^{(1)},tp_t^{(i)})+w(t\Lambda_{t,-}^{(1)}z+tp_t^{(1)})\right)+O(t^4|z|)
  \\&=-\frac{\rho_t^{(i)}}{2\pi}
  \frac{(z+\Lambda_{t,+}^{(1)}(p_t^{(1)}-p_t^{(i)}))}{|z+\Lambda_{t,+}^{(1)}(p_t^{(1)}-p_t^{(i)})|^2}+O(t^2\|\nabla\widetilde{\phi}_t\|_{L^\infty(M\setminus B_{2R_0t}(tp_t^{(1)}))})+O(t^4|z|)\\&+O(t^2)\nabla_x\left(\rho_t^{(i)}R(x,tp_t^{(i)})+w(x)\right)\Big|_{x=t\Lambda_{t,-}^{(1)}z+ tp_t^{(1)}}.\end{aligned}
\end{equation}
In view of Lemma \ref{lem_diff_p_t} and Theorem D, we see that there are $a_t^{(i)}\in\mathbb{R}^2$ such that $a_t^{(1)}=0$, $|a_t^{(2)}|=O(t\ln t)$, and
\begin{equation}\label{gradtilv2}
  \begin{aligned}  \nabla_z\widetilde{v}_t^{(i)}(z) &= -4\frac{z+a_t^{(i)}}{|z+a_t^{(i)}|^2}+O\left(t^2\nabla_x\left(\rho_t^{(1)}R(x,tp_t^{(1)})+w(x)\right)\Big|_{x=tp_t^{(1)}}\right)\\& +O(t^3\ln t) +O(t^4|z|)\ \ \ \ \ \textrm{for}\ \ \ 2R_0\Lambda_{t,+}^{(1)}\le |z|\le \frac{r_0}{t}\Lambda_{t,+}^{(1)}.\end{aligned}
\end{equation}
In view of \eqref{wemayassume} and Theorem D, we have $\nabla_x(\rho_t^{(1)}R(x,tp_t^{(1)})+w(x))\Big|_{x=tp_t^{(1)}}=O(t^2\ln t)$, and get that if $2R_0\Lambda_{t,+}^{(1)}\le |z|\le \frac{r_0}{t}\Lambda_{t,+}^{(1)}$, then
\begin{equation}\label{gradtilv3}
  \begin{aligned} \nabla_z\widetilde{v}_t^{(i)}(z)&= -4\frac{(z+a_t^{(i)})}{|z+a_t^{(i)}|^2}+O(t^{3}\ln t)+O(t^4|z|)
  \\&= -4\frac{z }{|z |^2}+O(\frac{|a_t^{(i)}|}{|z|^2})+O(t^3\ln t)+O(t^4|z|)
  \\&= -4\frac{z }{|z |^2} +O(t^3\ln t)+O(t^4|z|).\end{aligned}
\end{equation}
From \eqref{int0}, we recall $\int_M\zeta_t {d}v_g=0$.
Together with Green's representation formula, we have
\begin{equation*}
  \zeta_t(x)=\int_M\rho he^{-G_t}\frac{(e^{\widetilde{u}_t^{(1)}}-e^{\widetilde{u}_t^{(2)}})}{\|u_t^{(1)}-u_t^{(2)}\|_{L^\infty(M)}}G(x,y){d}y,
\end{equation*}
and thus  \begin{equation}\label{zetagr}\begin{aligned}
   \nabla_x\zeta_t(x) &=\int_{M\setminus B_{2R_0t}(tp_t^{(1)})}\rho he^{-G_t}\frac{(e^{\widetilde{u}_t^{(1)}}-e^{\widetilde{u}_t^{(2)}})}{\|u_t^{(1)}-u_t^{(2)}\|_{L^\infty(M)}}\nabla_xG(x,y){d}y
  \\&+\nabla_xG(x,tp_t^{(1)})\int_{B_{2R_0t}(tp_t^{(1)})}\rho he^{-G_t}\frac{(e^{\widetilde{u}_t^{(1)}}-e^{\widetilde{u}_t^{(2)}})}{\|u_t^{(1)}-u_t^{(2)}\|_{L^\infty(M)}} {d}y
  \\&+\int_{B_{2R_0t}(tp_t^{(1)})}\rho he^{-G_t}\frac{(e^{\widetilde{u}_t^{(1)}}-e^{\widetilde{u}_t^{(2)}})}{\|u_t^{(1)}-u_t^{(2)}\|_{L^\infty(M)}}(\nabla_xG(x,y)-\nabla_xG(x,tp_t^{(1)})) {d}y
  \\&:=I+II+III.\end{aligned}
\end{equation}
  From Lemma \ref{lem_equalb0}-(ii) and Lemma \ref{lem_intzetathree}-(iii), we see that if $x\in M\setminus B_{2R_0t}(tp_t^{(1)})$,
  then \begin{equation}\begin{aligned}\label{I}
         I&=\int_{M\setminus B_{2R_0t}(tp_t^{(1)})}\nabla_xG(x,y)\frac{\rho he^{-G_t+\widetilde{u}_t^{(1)}}}{\|u_t^{(1)}-u_t^{(2)}\|_{L^\infty(M)}}\\&\times\Big(u_t^{(1)}-u_t^{(2)}-\frac{\int_M he^{u_t^{(1)}-G_t}(u_t^{(1)}-u_t^{(2)}){d}v_g}{\int_Mhe^{u_t^{(1)}-G_t}{d}v_g} +O( \|u_t^{(1)}-u_t^{(2)}\|_{L^\infty(M)}^2)\Big){d}y
         \\&=\int_{M\setminus B_{2R_0t}(tp_t^{(1)})}\nabla_xG(x,y)\rho he^{-G_t+\widetilde{u}_t^{(1)}} \Big(\zeta_t-\frac{\int_M he^{u_t^{(1)}-G_t}\zeta_t{d}v_g}{\int_Mhe^{u_t^{(1)}-G_t}{d}v_g}+O( \|u_t^{(1)}-u_t^{(2)}\|_{L^\infty(M)})\Big){d}y\\&=o(1)\ \ \ \ \textrm{as}\ \ t\to0.
      \end{aligned} \end{equation}
    From Lemma \ref{lem_intzetathree}-(ii), we have  \[A_t=-\int_{B_{2R_0t}(tp_t^{(1)})}\Delta\zeta_t{d}x=-\int_{B_{2R_0\Lambda_{t,+}^{(1)}}(0)}\Delta\widetilde{\zeta}_t{d}z=O(t\ln t).\] Then
      we see that  if $x\in M\setminus B_{2R_0t}(tp_t^{(1)})$,
  then \begin{equation}\begin{aligned}\label{II}
         II&= \nabla_x G(x,tp_t^{(1)})A_t=\left\{-\frac{1}{2\pi}\frac{(x-tp_t^{(1)})1_{B_{r_0}(tp_t^{(1)})}(x)}{|x-tp_t^{(1)}|^2}+O(1)\right\}A_t,
                 \end{aligned} \end{equation}
                 where \begin{equation}
\label{1br0}
1_{B_{r_0}(tp_t^{(1)})}(x)=\left\{\begin{array}{l}
1\ \ \textrm{if}\ \ x\in  B_{r_0}(tp_t^{(1)}),\\
0\ \ \textrm{if}\ \ x\in M\setminus B_{r_0}(tp_t^{(1)}).
\end{array}\right.
\end{equation}
      Now we also see that if $x\in M\setminus B_{2R_0t}(tp_t^{(1)})$,
  then \begin{equation}\begin{aligned}\label{III9}
          III &=\int_{B_{2R_0t}(tp_t^{(1)})} \frac{\rho he^{-G_t}(e^{\widetilde{u}_t^{(1)}}-e^{\widetilde{u}_t^{(2)}})}{2\pi\|u_t^{(1)}-u_t^{(2)}\|_{L^\infty(M)}}\Big(\frac{x-tp_t^{(1)}}{|x-tp_t^{(1)}|^2}-\frac{x-y}{|x-y|^2}\Big)1_{B_{r_0}(tp_t^{(1)})}(x){d}y
         +o(1)
         \\&=-\frac{1}{2\pi}\int_{B_{2R_0t}(tp_t^{(1)})}\Delta\zeta_t\Big(\frac{x-tp_t^{(1)}}{|x-tp_t^{(1)}|^2}-\frac{x-y}{|x-y|^2}\Big)1_{B_{r_0}(tp_t^{(1)})}(x){d}y +o(1),
      \end{aligned} \end{equation}
and
      \begin{equation}\begin{aligned}\label{III}
         &\int_{B_{2R_0t}(tp_t^{(1)})}\Delta\zeta_t\Big(\frac{x-tp_t^{(1)}}{|x-tp_t^{(1)}|^2}-\frac{x-y}{|x-y|^2}\Big)1_{B_{r_0}(tp_t^{(1)})}(x){d}y
         \\&=\left( \int_{B_{R_0\Lambda_{t,+}^{(1)}}(0)}+\int_{B_{2R_0\Lambda_{t,+}^{(1)}}(0)\setminus B_{R_0\Lambda_{t,+}^{(1)}}(0)}\right)\Delta_z\widetilde{\zeta}_t(z) \Big(\frac{x-tp_t^{(1)}}{|x-tp_t^{(1)}|^2}
         -\frac{x-tp_t^{(1)}-t\Lambda_{t,-}^{(1)}z}{|x-tp_t^{(1)}-t\Lambda_{t,-}^{(1)}z|^2}\Big)1_{B_{r_0}(tp_t^{(1)})}(x){d}z
         \\&= \int_{B_{R_0\Lambda_{t,+}^{(1)}}(0)}|\Delta_z\widetilde{\zeta}_t(z)|O \Big(\frac{t\Lambda_{t,-}^{(1)}|z|}{|x-tp_t^{(1)}|^2}\Big)1_{B_{r_0}(tp_t^{(1)})}(x){d}z
         \\&
+\int_{B_{2R_0\Lambda_{t,+}^{(1)}}(0)\setminus B_{R_0\Lambda_{t,+}^{(1)}}(0)}\frac{O(1)}{|z|^4}\left(\frac{1}{|x-tp_t^{(1)}|}+\frac{1}{|x-tp_t^{(1)}-t\Lambda_{t,-}^{(1)}z|}\right) 1_{B_{r_0}(tp_t^{(1)})}(x){d}z         \\&=\int_{B_{R_0\Lambda_{t,+}^{(1)}}(0)}O(\frac{1}{(1+|z|^2)^2})\Big(\frac{t^2|z|}{|x-tp_t^{(1)}|^2}\Big)1_{B_{r_0}(tp_t^{(1)})}(x){d}z
\\&+\int_{B_{2R_0\Lambda_{t,+}^{(1)}}(0)\setminus B_{R_0\Lambda_{t,+}^{(1)}}(0)}\frac{O(1)}{|z|^4}\left(O(t^{-1})+\frac{O(t^{-2})}{\Big|\frac{x-tp_t^{(1)}}{t\Lambda_{t,-}^{(1)}}-z\Big|}\right) 1_{B_{r_0}(tp_t^{(1)})}(x){d}z
          \\&=O(\frac{t^21_{B_{r_0}(tp_t^{(1)})}(x)}{|x-tp_t^{(1)}|^2}) +o(1)\ \ \textrm{as}\ \ t\to0.
      \end{aligned} \end{equation}
From \eqref{I}-\eqref{III}, we see that if $x\in M\setminus B_{2R_0t}(tp_t^{(1)})$,
\begin{equation}\label{final_zetagr}\begin{aligned}
  \nabla_x\zeta_t(x)&=-\frac{A_t}{2\pi}\frac{(x-tp_t^{(1)})1_{B_{r_0}(tp_t^{(1)})}(x)}{|x-tp_t^{(1)}|^2}
 +O(\frac{t^21_{B_{r_0}(tp_t^{(1)})}(x)}{|x-tp_t^{(1)}|^2}) +o(1)\ \ \textrm{as}\ \ t\to0.\end{aligned}
\end{equation}
Here we also note that if $2R_0\Lambda_{t,+}^{(1)}\le |z|\le \frac{2r_0\Lambda_{t,+}^{(1)}}{t}$, then
\begin{equation}\label{diffgrad}\begin{aligned}
 \frac{\nabla_z(\widetilde{v}_t^{(1)}(z)-\widetilde{v}_t^{(2)}(z))}{\|u_t^{(1)}-u_t^{(2)}\|_{L^\infty(M)}} &=\frac{\nabla_z(\widetilde{u}_t^{(1)}(t\Lambda_{t,-}^{(1)}z+tp_t^{(1)})
-\widetilde{u}_t^{(2)}(t\Lambda_{t,-}^{(1)}z+tp_t^{(1)}))}{\|u_t^{(1)}-u_t^{(2)}\|_{L^\infty(M)}}
\\&=\nabla_z\zeta_t(t\Lambda_{t,-}^{(1)}z+tp_t^{(1)})
\\&=\nabla_z\widetilde{\zeta}_t(z)=t\Lambda_{t,-}^{(1)}\left(-\frac{A_t}{2\pi}\frac{1}{t\Lambda_{t,-}^{(1)}}\frac{z}{|z|^2}+O(1)\right)
=-\frac{A_t}{2\pi}\frac{z}{|z|^2}+O(t^2),
\end{aligned}
\end{equation}and
\begin{equation}\label{diffgrad2}\begin{aligned}
&\nabla_z\widetilde{\zeta}_t(z)\Big|_{z\in \partial B_{2R_0\Lambda_{t,+}^{(1)}}(0) } =t\Lambda_{t,-}^{(1)}\nabla_x\zeta_t(x)\Big|_{x\in \partial B_{2R_0t}(tp_t^{(1)})}.
\end{aligned}
\end{equation}
By \eqref{eqtilv3}, we obtain for $r_{t,R_0}=2R_0\Lambda_{t,+}^{(1)}$,
\begin{equation}\label{eq_tilv}
  \begin{aligned} &\frac{1}{2}\int_{\partial B_{r_{t,R_0}}(0)}\nabla (\widetilde{\zeta}_t)\cdot \nabla (\widetilde{v}_t^{(1)}+\widetilde{v}_t^{(2)})|z|{d}\sigma
 -\int_{\partial  B_{r_{t,R_0}}(0)}\frac{(\nabla (\widetilde{\zeta}_t)\cdot z)(\nabla (\widetilde{v}_t^{(1)}+\widetilde{v}_t^{(2)})\cdot z)}{|z|}{d}\sigma
\\&=\int_{\partial  B_{r_{t,R_0}}(0)}\frac{(\Lambda_{t,-}^{(1)})^2h_t(\Lambda_{t,-}^{(1)}z+p_t^{(1)}) }{\|u_t^{(1)}-u_t^{(2)}\|_{L^\infty(M)}} e^{\widetilde{v}_t^{(1)}(z)}(1-e^{\widetilde{v}_t^{(2)}(z)-\widetilde{v}_t^{(1)}(z)})|z|{d}\sigma\\&-\int_{ B_{r_{t,R_0}}(0)}\frac{(\Lambda_{t,-}^{(1)})^2h_t(\Lambda_{t,-}^{(1)}z+p_t^{(1)}) }{\|u_t^{(1)}-u_t^{(2)}\|_{L^\infty(M)}} e^{\widetilde{v}_t^{(1)}(z)}(1-e^{\widetilde{v}_t^{(2)}(z)-\widetilde{v}_t^{(1)}(z)})  \left(2+\nabla_z\ln h_t(\Lambda_{t,-}^{(1)}z+p_t^{(1)})\cdot z\right){d}z.\end{aligned}
\end{equation}
We see from \eqref{gradtilv3} and \eqref{diffgrad} that
\begin{equation}\label{eq_tilvl}
  \begin{aligned}  \mbox{(LHS) of } \eqref{eq_tilv}
   &=\frac{1}{2}\int_{\partial B_{r_{t,R_0}}(0)}|z|\left(-8\frac{z}{|z|^2}+O(t^3\ln t)\right)\cdot \left(-\frac{A_t}{2\pi}\frac{z}{|z|^2}+O(t^2)\right) {d}\sigma
  \\&-\int_{\partial B_{r_{t,R_0}}(0)}\frac{1}{|z|}\left(-\frac{A_t}{2\pi}+O(t)\right)\left(-8+O(t^2\ln t)\right) d\sigma
  \\&=\frac{1}{2}\int_{\partial B_{r_{t,R_0}}(0)}|z|\left(-8\frac{z}{|z|^2}+O(t^3\ln t)\right)\cdot \left(-\frac{A_t}{2\pi}\frac{z}{|z|^2}+O(t^2)\right) {d}\sigma
  \\&-\int_{\partial B_{r_{t,R_0}}(0)}\frac{1}{|z|}\left(\frac{4A_t }{\pi}+O(t)\right) d\sigma
  \\&=\frac{1}{2}\int_{\partial B_{r_{t,R_0}}(0)}|z|\left(\frac{4A_t}{\pi|z|^2}+O(t^{3}) \right) {d}\sigma
-\int_{\partial B_{r_{t,R_0}}(0)}\frac{1}{|z|}\left(\frac{4A_t }{\pi}\right) d\sigma+O(t)
  \\&=-4A_t +O(t).
  \end{aligned}
\end{equation}
We also see from \eqref{difftilv0}-\eqref{eqtilv} and \eqref{tilvt} that
\begin{equation}\label{eq_tilvr}
  \begin{aligned}   \mbox{(RHS) of } \eqref{eq_tilv} &=
\int_{B_{r_{t,R_0}}(0)}\frac{2(\Delta \widetilde{v}_t^{(1)}-\Delta \widetilde{v}_t^{(2)})}{\|u_t^{(1)}-u_t^{(2)}\|_{L^\infty(M)}} +\frac{O(t|z|)}{(1+|z|^2)^2} dz+\int_{\partial B_{r_{t,R_0}}(0)}O(\frac{1}{|z|^3}){d}\sigma
\\&=2\int_{B_{r_{t,R_0}}(0)} \Delta_z \widetilde{\zeta}_t(z) dz+O(t)\\&=2\int_{B_{2R_0t}(tp_t^{(1)})} \Delta_x \zeta_t(x) dx+O(t)=-2A_t+O(t).\end{aligned}
\end{equation}
By \eqref{eq_tilvl}-\eqref{eq_tilvr}, we obtain
$A_t=O(t),$ and complete the proof of Lemma \ref{lem_estat}.
\end{proof}
For any function $f$, we denote
\begin{equation}\label{deril}
  D_l f(z)=\frac{\partial f (z)}{\partial z_l}\ \ \textrm{for} \ \ l=1,2.
\end{equation}
\begin{lemma}\label{lem_bj0}
  (i) $b_1=b_2=0$,

  (ii) $\widetilde{\zeta}_t(z)\to0,$  $\zeta_t(t\Lambda_{t,-}^{(1)}z+tp_t^{(1)})\to0$ in $C_{\textrm{loc}}^0(\mathbb{R}^2)$ as $t\to0$,

  (iii) $\lim_{t\to0}\Bigg(\int_{B_{2R_0\Lambda_{t,+}^{(1)}}(0)}(\ln |z|)\Delta\widetilde{\zeta}_t{d}z\Bigg)=0$.
\end{lemma}
\begin{proof} (i)
  We have
  \begin{equation}\label{div1}
    \begin{aligned}
   & \mbox{div}\left(\nabla\widetilde{\zeta}_tD_l\widetilde{v}_t^{(i)}+\nabla \widetilde{v}_t^{(i)}D_l\widetilde{\zeta}_t-\nabla\widetilde{\zeta}_t\cdot\nabla\widetilde{v}_t^{(i)}e_l\right)
    =\Delta\widetilde{\zeta}_tD_l\widetilde{v}_t^{(i)}+\Delta\widetilde{v}_t^{(i)}D_l\widetilde{\zeta}_t
    \\&=\frac{\Delta(\widetilde{v}_t^{(1)}-\widetilde{v}_t^{(2)})}{\|u_t^{(1)}-u_t^{(2)}\|_{L^\infty(M)}}D_l\widetilde{v}_t^{(i)}+\frac{\Delta\widetilde{v}_t^{(i)}D_l(\widetilde{v}_t^{(1)}-\widetilde{v}_t^{(2)})}{\|u_t^{(1)}-u_t^{(2)}\|_{L^\infty(M)}}
    \\&=-\frac{(\Lambda_{t,-}^{(1)})^2h_t(\Lambda_{t,-}^{(1)}z+p_t^{(1)})(e^{\widetilde{v}_t^{(1)}}-e^{\widetilde{v}_t^{(2)}})D_l\widetilde{v}_t^{(i)}}{\|u_t^{(1)}-u_t^{(2)}\|_{L^\infty(M)}}
     -\frac{(\Lambda_{t,-}^{(1)})^2h_t(\Lambda_{t,-}^{(1)}z+p_t^{(1)}) e^{\widetilde{v}_t^{(i)}} D_l(\widetilde{v}_t^{(1)}-\widetilde{v}_t^{(2)})}{\|u_t^{(1)}-u_t^{(2)}\|_{L^\infty(M)}}
    \\&=-\mbox{div}\left(\frac{(\Lambda_{t,-}^{(1)})^2h_t(\Lambda_{t,-}^{(1)}z+p_t^{(1)})
    (e^{\widetilde{v}_t^{(1)}}-e^{\widetilde{v}_t^{(2)}})e_l}{\|u_t^{(1)}-u_t^{(2)}\|_{L^\infty(M)}}  \right)
    \\&+\Bigg[\frac{(\Lambda_{t,-}^{(1)})^2h_t(\Lambda_{t,-}^{(1)}z+p_t^{(1)})
    (e^{\widetilde{v}_t^{(1)}}-e^{\widetilde{v}_t^{(2)}})}{\|u_t^{(1)}-u_t^{(2)}\|_{L^\infty(M)}}  D_l\left((-1)^{i}(\widetilde{v}_t^{(1)}-\widetilde{v}_t^{(2)})+\ln h_t(\Lambda_{t,-}^{(1)}z+p_t^{(1)})\right)\Bigg].
    \end{aligned}
  \end{equation}
  For any constant $R\ge R_0$, let $r_{t,R}=2R\Lambda_{t,+}^{(1)}$. Then  \eqref{div1} implies
   \begin{equation}\label{intdiv1}
    \begin{aligned}
   & \int_{\partial B_{r_{t,R}}(0)} \left(2\nabla\widetilde{\zeta}_tD_l \widetilde{v}_t^{(1)} +2\nabla \widetilde{v}_t^{(2)}D_l\widetilde{\zeta}_t-\nabla\widetilde{\zeta}_t\cdot(\nabla\widetilde{v}_t^{(1)}+\nabla\widetilde{v}_t^{(2)})e_l\right)\cdot\frac{z}{|z|} d\sigma
    \\&=-2\int_{\partial B_{r_{t,R}}(0)}\left(\frac{(\Lambda_{t,-}^{(1)})^2h_t(\Lambda_{t,-}^{(1)}z+p_t^{(1)})
    (e^{\widetilde{v}_t^{(1)}}-e^{\widetilde{v}_t^{(2)}})e_l}{\|u_t^{(1)}-u_t^{(2)}\|_{L^\infty(M)}}  \right)\cdot\frac{z}{|z|}d\sigma
    \\&+2\Bigg[\int_{  B_{r_{t,R}}(0)}\frac{(\Lambda_{t,-}^{(1)})^2h_t(\Lambda_{t,-}^{(1)}z+p_t^{(1)})
    (e^{\widetilde{v}_t^{(1)}}-e^{\widetilde{v}_t^{(2)}})}{\|u_t^{(1)}-u_t^{(2)}\|_{L^\infty(M)}}
 D_l\left( \ln h_t(\Lambda_{t,-}^{(1)}z+p_t^{(1)})\right)dz\Bigg]. \end{aligned}
  \end{equation}
    By \eqref{final_zetagr} and Lemma \ref{lem_estat}, we have  if   $x\in M\setminus B_{2R_0t}(tp_t^{(1)})$, then
  \begin{equation}\label{gratilv}
    \begin{aligned}\nabla_x\zeta_t(x)&=-\frac{A_t(x-tp_t^{(1)})1_{B_{r_0}(tp_t^{(1)})}(x)}{2\pi|x-tp_t^{(1)}|^2} +\frac{O(t^2)1_{B_{r_0}(tp_t^{(1)})}(x)}{|x-tp_t^{(1)}|^2}+o(1)
     = \frac{O(t)1_{B_{r_0}(tp_t^{(1)})}(x)}{|x-tp_t^{(1)}|}+o(1),\end{aligned}
  \end{equation}
  which implies \begin{equation}\label{gratilv2}
    \begin{aligned}\nabla_z\widetilde{\zeta}_t(z)&=t\Lambda_{t,-}^{(1)}\nabla_x\zeta_t(x)\Big|_{x=t\Lambda_{t,-}^{(1)}z+tp_t^{(1)}}
     = \frac{O(t)}{|z|}1_{B_{\frac{r_0\Lambda_{t,+}^{(1)}}{t}}(0)}(z)+o(t^2),\end{aligned}
  \end{equation}for $2R_0\Lambda_{t,+}^{(1)}\le |z|\le \frac{r_0\Lambda_{t,+}^{(1)}}{t}$.\\
  Therefore, in view of \eqref{gradtilv3} and \eqref{gratilv2} we get that  \begin{equation}\label{eq_zetal}
  \begin{aligned}  \mbox{(LHS) of } \eqref{intdiv1} &=|z|\left(O(t^3\ln t)+O(t^4|z|)+\frac{O(1)}{|z|}\right)\left( O(\frac{t}{|z|})+o(t^2)\right)\Big|_{z\in \partial B_{2R\Lambda_{t,+}^{(1)}}(0)}.
  \\&= O(t^{4}(\ln t) R  )+O(\frac{t^2}{R})+o(t^4)R^2+o(t^2).\end{aligned}
\end{equation}
To estimate (RHS) of \eqref{intdiv1}, by the change of variables $x=t\Lambda_{t,-}^{(1)}z+tp_t^{(1)}$, we see that if $|z|=2R\Lambda_{t,+}^{(1)}\ge 2R_0\Lambda_{t,+}^{(1)}$, then Theorem D implies
\begin{equation}\label{eq_zetar1}
  \begin{aligned}  &-2\int_{\partial B_{r_{t,R}}(0)}\left(\frac{ (\Lambda_{t,-}^{(1)})^2h_t(\Lambda_{t,-}^{(1)}z+p_t^{(1)})
    (e^{\widetilde{v}_t^{(1)}}-e^{\widetilde{v}_t^{(2)}})e_l}{\|u_t^{(1)}-u_t^{(2)}\|_{L^\infty(M)}}  \right)\cdot\frac{z}{|z|}d\sigma(z)
  \\&=\frac{-2\Lambda_{t,+}^{(1)}}{t}\int_{\partial B_{2Rt}(tp_t^{(1)})} \frac{\rho(\Lambda_{t,-}^{(1)})^2\left(h(ty) |y-\underline{e}|^2|y+\underline{e}|^2e^{-R_t(ty)}\right)\Big|_{y=\Lambda_{t,-}^{(1)}z+p_t^{(1)}}
      }{\|u_t^{(1)}-u_t^{(2)}\|_{L^\infty(M)}}  \\& \times\frac{ (e^{\widetilde{u}_t^{(1)}}-e^{\widetilde{u}_t^{(2)}})t^6e_l\cdot(x-tp_t^{(1)})}{|x-tp_t^{(1)}|}d\sigma(x)
     \\&=O(1)\left(\int_{\partial B_{2Rt}(tp_t^{(1)})}\frac{t^2h(x) e^{-G_t(x)}
      (e^{\widetilde{u}_t^{(1)}}-e^{\widetilde{u}_t^{(2)}})e_l}{\|u_t^{(1)}-u_t^{(2)}\|_{L^\infty(M)}} \cdot\frac{x-tp_t^{(1)}}{|x-tp_t^{(1)}|}d\sigma(x)\right)
 =O(t^3R).     \end{aligned}
\end{equation}
Let $x=t\Lambda_{t,-}^{(1)}z+tp_t^{(1)}$ and $y=\Lambda_{t,-}^{(1)}z+p_t^{(1)}$. \\ Then $r_0\ge |x-tp_t^{(1)}|=t\Lambda_{t,-}^{(1)}|z|=2Rt\ge 2R_0t$ implies
$\frac{r_0}{t}\ge |y-p_t^{(1)}|=\Lambda_{t,-}^{(1)}|z|=2R\ge 2R_0$. So we see that if $2R_0\le |y-p_t^{(1)}|\le \frac{r_0}{t}$, then
\begin{equation}\label{bddh}
  \begin{aligned}\nabla_y\ln h_t(y)&=\nabla_y(\ln \rho h(ty) e^{-R_t(ty)+\psi(ty)}+2\ln |y-\underline{e}|+2\ln |y+\underline{e}|)
  \\&=t\nabla_{ty}\ln (\rho h(ty)e^{-R_t(ty)+\psi(ty)})+O(\frac{1}{|y|})=O(1).
  \end{aligned}
\end{equation}
In view of Lemma \ref{lem_equalb0}-(ii)  and  Lemma \ref{lem_intzetathree}-(iii),  we have
\begin{equation}\label{toge}
  \zeta_t-\frac{ \int_Mhe^{u_t^{(1)}-G_t}\zeta_t{d}v_g}{\int_Mhe^{u_t^{(1)}-G_t}{d}v_g}=o(1)\ \ \textrm{in}\ \ M\setminus B_{\frac{tR_0}{2}}(tp_t^{(1)}).
\end{equation}
Together with \eqref{difftilv}, we  see that \begin{equation}\label{eq_zetar3}
    \begin{aligned}
  & \frac{1}{\|u_t^{(1)}-u_t^{(2)}\|_{L^\infty(M)}}\int_{  B_{2R\Lambda_{t,+}^{(1)}}(0)\setminus B_{2R_0\Lambda_{t,+}^{(1)}}(0)}(\Lambda_{t,-}^{(1)})^2h_t(\Lambda_{t,-}^{(1)}z+p_t^{(1)})
    (e^{\widetilde{v}_t^{(1)}}-e^{\widetilde{v}_t^{(2)}}) D_l\left( \ln h_t(\Lambda_{t,-}^{(1)}z+p_t^{(1)})\right)
    dz
    \\&=O(\frac{t^3}{\|u_t^{(1)}-u_t^{(2)}\|_{L^\infty(M)}})\int_{  B_{2R\Lambda_{t,+}^{(1)}}(0)\setminus B_{2R_0\Lambda_{t,+}^{(1)}}(0)}  h_t(\Lambda_{t,-}^{(1)}z+p_t^{(1)})
    (e^{\widetilde{v}_t^{(1)}}-e^{\widetilde{v}_t^{(2)}})D_{y_l} \ln h_t(y)\Big|_{y=\Lambda_{t,-}^{(1)}z+p_t^{(1)}}
    dz
    \\&= O(t^3) \int_{  B_{2R\Lambda_{t,+}^{(1)}}(0)\setminus B_{2R_0\Lambda_{t,+}^{(1)}}(0)} h(t\Lambda_{t,-}^{(1)}z+tp_t^{(1)})e^{-R_t(t\Lambda_{t,-}^{(1)}z+tp_t^{(1)})}\\&\times |\Lambda_{t,-}^{(1)}z+p_t^{(1)}-\underline{e}|^2|\Lambda_{t,-}^{(1)}z+p_t^{(1)}+\underline{e}|^2 O(t^6)e^{\widetilde{u}_t^{(1)}(t\Lambda_{t,-}^{(1)}z+tp_t^{(1)})}
    \\&\times\frac{(\widetilde{u}_t^{(1)}(t\Lambda_{t,-}^{(1)}z+tp_t^{(1)})-
    \widetilde{u}_t^{(2)}(t\Lambda_{t,-}^{(1)}z+tp_t^{(1)})+O(\|u_t^{(1)}-u_t^{(2)}\|_{L^\infty(M)}^2)}{\|u_t^{(1)}-u_t^{(2)}\|_{L^\infty(M)}}
    dz
     \\&= O(t) \int_{  B_{2Rt}(tp_t^{(1)})\setminus B_{2R_0t}(tp_t^{(1)})}h(x)|x-t\underline{e}|^2|x+t\underline{e}|^2
     e^{-R_t(x)} e^{\widetilde{u}_t^{(1)}(x)}\\&\times (\zeta_t(x)-\frac{ \int_Mhe^{u_t^{(1)}-G_t}\zeta_t{d}v_g}{\int_Mhe^{u_t^{(1)}-G_t}{d}v_g}+O(\|u_t^{(1)}-u_t^{(2)}\|_{L^\infty(M)}))
    dx
     \\&= O(t) \int_{  B_{2Rt}(tp_t^{(1)})\setminus B_{2R_0t}(tp_t^{(1)})}o(1)
    dx
    =o(t^3)R^2.\end{aligned}
  \end{equation}
  To estimate $2\int_{B_{2R_0\Lambda_{t,+}^{(1)}}(0)}\frac{(\Lambda_{t,-}^{(1)})^2h_t(\Lambda_{t,-}^{(1)}z+p_t^{(1)})
    (e^{\widetilde{v}_t^{(1)}}-e^{\widetilde{v}_t^{(2)}})D_l\left( \ln h_t(\Lambda_{t,-}^{(1)}z+tp_t^{(1)})\right)}{\|u_t^{(1)}-u_t^{(2)}\|_{L^\infty(M)}}
    dz $, we note that if $|z|\le 2R_0 \Lambda_{t,+}^{(1)}$, then there is $\theta\in(0,1)$ such that
    \begin{equation}\label{dl}
      \begin{aligned}
     & D_{z_l}[\ln h_t(\Lambda_{t,-}^{(1)} z+p_t^{(1)})]
     =[D_{y_l}\ln h_t(y)]\Big|_{y=\Lambda_{t,-}^{(1)}z +p_t^{(1)}}\Lambda_{t,-}^{(1)}
     \\&=\Bigg[D_{y_l}\ln h_t(y)\Big|_{y=p_t^{(1)}}+\sum_{k=1}^2D_{y_k}D_{y_l}\ln h_t(y)\Big|_{y=p_t^{(1)}}\Lambda_{t,-}^{(1)}z_k  +O(D^3\ln h_t\Big|_{y=\theta\Lambda_{t,-}^{(1)}z +p_t^{(1)}}t^2|z|^2)\Bigg]\Lambda_{t,-}^{(1)}
    \\&=\Bigg[D_{y_l}\ln h_t(y)\Big|_{y=p_t^{(1)}}+\sum_{k=1}^2D_{y_k}D_{y_l}\ln h_t(y)\Big|_{y=p_t^{(1)}}\Lambda_{t,-}^{(1)}z_k +O(t^2|z|^2)\Bigg]\Lambda_{t,-}^{(1)}.
    \end{aligned}
          \end{equation}
Moreover, by using the proof of Lemma \ref{lem_diff_p_t} and  \eqref{wemayassume}, we get that
\begin{equation}\label{rhop}
  \begin{aligned}\nabla_{y}\ln h_t(y)\Big|_{y=p_t^{(1)}}&=-t \nabla_x G_{*,t}^{(1)}(x)\Big|_{x=tp_t^{(1)}}+O(t\|\widetilde{\phi}_t\|_*+t^2\ln t)
  =O(t^2\ln t).
  \end{aligned}
\end{equation}Now we obtain\begin{equation}\label{dl0}
      \begin{aligned}
     & D_{z_l}[\ln h_t(\Lambda_{t,-}^{(1)}z +p_t^{(1)})]
        =\Bigg[ \sum_{k=1}^2D_{y_k}D_{y_l}\ln h_t(y)\Big|_{y=p_t^{(1)}} \Lambda_{t,-}^{(1)}z_k +O(t^2\ln t)+O(t^2|z|^2)\Bigg]\Lambda_{t,-}^{(1)}.
    \end{aligned}
          \end{equation}
    Together with \eqref{tilvt} and \eqref{difftilv0}, we see that
  \begin{equation}\label{eq_zetar4}
    \begin{aligned}
     &2\int_{     B_{2R_0\Lambda_{t,+}^{(1)}}(0)}\frac{(\Lambda_{t,-}^{(1)})^2h_t(\Lambda_{t,-}^{(1)}z+p_t^{(1)})
    (e^{\widetilde{v}_t^{(1)}}-e^{\widetilde{v}_t^{(2)}})}{\|u_t^{(1)}-u_t^{(2)}\|_{L^\infty(M)}}
     D_l\left( \ln h_t(\Lambda_{t,-}^{(1)}z+p_t^{(1)})\right)dz
        \\&=2\int_{ B_{2R_0\Lambda_{t,+}^{(1)}}(0)}\frac{(\Lambda_{t,-}^{(1)})^2h_t(\Lambda_{t,-}^{(1)}z+p_t^{(1)})
    e^{\widetilde{v}_t^{(1)}}(1-e^{\widetilde{v}_t^{(2)}-\widetilde{v}_t^{(1)}})}{\|u_t^{(1)}-u_t^{(2)}\|_{L^\infty(M)}}\\&\times[ \sum_{k=1}^2D_{y_k}D_{y_l}\ln h_t(y)\Big|_{y=p_t^{(1)}}\Lambda_{t,-}^{(1)}z_k +O(t^2\ln t)+O(t^2|z|^2)]\Lambda_{t,-}^{(1)}
    dz
    \\&=2\Bigg\{\int_{ B_{2R_0\Lambda_{t,+}^{(1)}}(0)}\frac{h_t(\Lambda_{t,-}^{(1)}z+p_t^{(1)})(1+|\widetilde{\eta}_t^{(1)}| +t^2|z|)}{C_t^{(1)}(1+|z+O(t^2)|^2)^2}
   (\widetilde{\zeta}_t(z)+O(\|u_t^{(1)}-u_t^{(2)}\|_{L^\infty(M)}))
 \\&\times[ \sum_{k=1}^2D_{y_k}D_{y_l}\ln h_t(y)\Big|_{y=p_t^{(1)}}\Lambda_{t,-}^{(1)}z_k +O(t^2\ln t)+O(t^2|z|^2)]\Lambda_{t,-}^{(1)}
    dz\Bigg\}.
    \end{aligned}\end{equation}
    We note from \eqref{rhop} that
    \begin{equation}\label{htexp}\begin{aligned}
    \frac{h_t(\Lambda_{t,-}^{(1)}z+p_t^{(1)})}{C_t^{(1)}}&=  \frac{8h_t(\Lambda_{t,-}^{(1)}z+p_t^{(1)})}{h_t(p_t^{(1)})} =8\Bigg(1+\frac{\nabla h_t(p_t^{(1)})}{h_t(p_t^{(1)})}\cdot \Lambda_{t,-}^{(1)}z +O(t^2|z|^2)\Bigg)\\&=8+O(t^2\ln t)+O(t^2|z|^2).
    \end{aligned}\end{equation}
    By using Lemma \ref{lem_est_zetanj} and Lemma \ref{lem_intzetathree}-(v), we have
    \begin{equation}\label{zetacon}
      \widetilde{\zeta}_t\to \sum_{i=1}^2\frac{b_iz_i}{1+|z|^2}
    \ \ \ \textrm{in} \ \ C_{\textrm{loc}}^0(\mathbb{R}^2).
    \end{equation} Together with Theorem E, we have for any $R>1$,
   \begin{equation} \label{eq_zetar5}
    \begin{aligned}
     &2\int_{  B_{2R_0\Lambda_{t,+}^{(1)}}(0)}\frac{(\Lambda_{t,-}^{(1)})^2h_t(\Lambda_{t,-}^{(1)}z+p_t^{(1)})
    (e^{\widetilde{v}_t^{(1)}}-e^{\widetilde{v}_t^{(2)}})}{\|u_t^{(1)}-u_t^{(2)}\|_{L^\infty(M)}}
  D_l\left( \ln h_t(\Lambda_{t,-}^{(1)}z+tp_t^{(1)})\right)dz
    \\
  &=16\int_{ B_{R}(0)}\frac{(1+O(t^2\ln t(|z|+1)^{\varepsilon})+O(t^2|z|^2))}{(1+|z|^2)^2}
    (\sum_{i=1}^2\frac{b_iz_i}{(1+|z|^2)}+o(1))
 \\&\times[ \sum_{k=1}^2D_{y_k}D_{y_l}\ln h_t(y)\Big|_{y=p_t^{(1)}}\Lambda_{t,-}^{(1)}z_k +O(t^2\ln t)+O(t^2|z|^2)]\Lambda_{t,-}^{(1)}
    dz
    \\&+\int_{ B_{2R_0\Lambda_{t,+}^{(1)}}(0)\setminus B_R(0)}\frac{t}{|z|^4}(O(t|z|)+O(t^2\ln t) )dz
   \\
  &=  8  (\Lambda_{t,-}^{(1)})^2 \sum_{k=1}^2D_{y_k}D_{y_l}\ln h_t(y)\Big|_{y=p_t^{(1)}}\int_{ B_{R}(0)}\frac{b_k|z|^2}{(1+|z|^2)^3}
 +\frac{O(t^2)}{R}+O(t^3\ln R)+O(t^4 R)+o(t^2),\end{aligned}
  \end{equation}for some $\varepsilon\in(0,1)$.
Therefore, we obtain  from \eqref{eq_zetar1}, \eqref{eq_zetar3}, and \eqref{eq_zetar5} that  \begin{equation}\label{eq_zetarf}
  \begin{aligned}& \mbox{(RHS) of } \eqref{intdiv1}\\&=8 (\Lambda_{t,-}^{(1)})^2 \sum_{k=1}^2D_{y_k}D_{y_l}\ln h_t(y)\Big|_{y=p_t^{(1)}}\left(\int_{ \mathbb{R}^2}\frac{b_k|z|^2}{(1+|z|^2)^3}{d}z+O(R^{-2})\right)
 +\frac{O(t^2)}{R}+O(t^3R^2)+o(t^2).\end{aligned}
\end{equation}
Since $\mbox{det}[\nabla^2(\ln h_t(p_t^{(1)}))]=-16+O(t)$,  \eqref{eq_zetal} and \eqref{eq_zetarf} imply  for $i=1,2$,
\begin{equation}\label{biest}
  |b_i|=O(\frac{1}{R})+o(1)\ \ \textrm{for any large }\ \ R\gg1.
\end{equation}
So we obtain
$b_1=b_2=0,$ and prove Lemma \ref{lem_bj0}-(i).

(ii) By Lemma \ref{lem_est_zetanj}, Lemma \ref{lem_intzetathree}-(v), and $b_1=b_2=0$, we have \[\widetilde{\zeta}_t(z)=\zeta_t(t\Lambda_{t,-}^{(1)}z+tp_{t}^{(1)})-\frac{ \int_Mhe^{u_t^{(1)}-G_t}\zeta_t{d}v_g}{\int_Mhe^{u_t^{(1)}-G_t}{d}v_g}\to0\ \ \textrm{in}\ \ \ C_{\textrm{loc}}^0(\mathbb{R}^2).\]
Together with Lemma \ref{lem_equalb0}-(ii), we obtain Lemma \ref{lem_bj0}-(ii).

(iii) In view of Lemma \ref{lem_est_zetanj}, \eqref{eq_zeta_final},  and $b_0=b_1=b_2=0$, it is easy to see that Lemma \ref{lem_bj0}-(iii) holds.
\end{proof}

Now we are going to complete the proof of Theorem \ref{thm_nonconcen}.\\
\textbf{Proof of Theorem \ref{thm_nonconcen}}

  Let $x_t^*$ be a maximum point of $\zeta_t$. So we have
  \begin{equation}\label{asymp0}
   |\zeta_t(x_t^*)|=1.
 \end{equation}    Then from Lemma \ref{lem_equalb0}, we have
 \begin{equation}\label{asymp1}
   \lim_{t\to0}x_t^*=0.
 \end{equation}Moreover, by     Lemma \ref{lem_intzetathree}-(iii) and Lemma \ref{lem_bj0}, we see that
  \begin{equation}\label{asymp2}\Lambda_{t,-}^{(1)}t\ll  s_t:=|x_t^* -tp_{t}^{(1)}|\le \frac{tR_0}{2}.
 \end{equation}
 Let     $\hat{\zeta}_t(\xi)=\zeta_t(s_t \xi+tp_{t}^{(1)})=\widetilde{\zeta}_t(\Lambda_{t,+}^{(1)} \frac{s_t}{t} \xi)+\frac{ \int_Mhe^{u_t^{(1)}-G_t}\zeta_t{d}v_g}{\int_Mhe^{u_t^{(1)}-G_t}{d}v_g}$. By \eqref{eq_zeta_final}, $\hat{\zeta}_t$ satisfies
 \begin{equation}\label{eq_tilzeta}\begin{aligned}
  0&= \Delta \hat{\zeta}_t+ O((\Lambda_{t,+}^{(1)})^2 \frac{s_t^2}{t^2})\frac{(1+\frac{s_t^2}{t^2}|\xi|^2)}{(1+(\Lambda_{t,+}^{(1)})^2 \frac{s_t^2}{t^2}|\xi|^2)^2}\ \ \ \textrm{for} \ \  |\xi|\le \frac{r_0}{s_t}.
\end{aligned} \end{equation}By \eqref{asymp0}, we have  \begin{equation}\label{tilzeta0}\Big|\hat{\zeta}_t\Big(\frac{x_t^*-tp_{t}^{(1)}}{s_t}\Big)\Big|=|\zeta_t(x_t^*)|=1.\end{equation} By \eqref{asymp2} and $|\hat{\zeta}_t|\le 1$,  we see that $\hat{\zeta}_t\to
\hat{\zeta}_0$ in any compact subset of $\mathbb{R}^2\setminus\{0\},$ where   $\hat{\zeta}_0$ satisfies $\Delta \hat{\zeta}_0=0$ in $\mathbb{R}^2\setminus\{0\}$.
Since $|\hat{\zeta}_0|\le 1$,  we have $\Delta \hat{\zeta}_0=0$ in $\mathbb{R}^2$. So $\hat{\zeta}_0$ is a constant.
From $ \frac{|x_t^*-tp_{t}^{(1)}|}{s_t}=1$ and \eqref{tilzeta0}, we have $\hat{\zeta}_0\equiv1$ or $\hat{\zeta}_0\equiv-1$. So  we have
\begin{equation}\label{compr}
  |\zeta_t(x)|\ge \frac{1}{2}\ \ \textrm{if}\ \ \ \frac{s_t}{2}\le |x-tp_{t}^{(1)}|\le s_t.
\end{equation}  By Lemma \ref{ref}-(ii), Lemma \ref{lem_bj0}, and Lemma \ref{lem_estat}, we see that if $\frac{s_t}{2}\le |x-tp_{t}^{(1)}|\le s_t$, then
\begin{equation*}\begin{aligned}
    \zeta_t(x)&=O\left(\ln t\right)\int_{B_{2R_0\Lambda_{t,+}^{(1)}}(0)}\Delta\widetilde{\zeta}_t{d}z+
    o(1)=o(1)\ \ \textrm{as} \ \ \ t\to0,\end{aligned}
          \end{equation*}which contradicts \eqref{compr}. So we complete the proof of Theorem \ref{thm_nonconcen}.\hfill$\Box$

\end{document}